\documentclass[10pt,journal]{IEEEtran}
\usepackage[bottom=0.68in,top=0.6in,left=0.625in,right=0.625in]{geometry}
\DeclareUnicodeCharacter{3000}{ }
\DeclareUnicodeCharacter{FF09}{  }
\DeclareUnicodeCharacter{FF08}{    }
\DeclareUnicodeCharacter{2161}{     }

\ifCLASSINFOpdf
  \usepackage[pdftex]{graphicx}
  \graphicspath{{../pdf/}{../jpeg/}}
  \DeclareGraphicsExtensions{.pdf,.jpeg,.png}
\else
  \usepackage[dvips]{graphicx}
  \graphicspath{{../eps/}}
  \DeclareGraphicsExtensions{.jpg}
\fi
\usepackage{enumitem}
\usepackage{epstopdf}
\usepackage{multirow}
\usepackage{float}
\usepackage[cmex10]{amsmath}
\usepackage {amssymb}
\usepackage{graphicx}
\usepackage{array}
\usepackage{mdwmath}
\usepackage{mdwtab}
\usepackage{eqparbox}

\usepackage{nomencl}
\usepackage{algorithm}
\usepackage{algorithmic}
\usepackage{mathtools}
\usepackage{makeidx}
\usepackage{ifthen}
\usepackage{subfigure}
\usepackage{caption}
\usepackage{color}
\usepackage{xcolor}

\usepackage{multirow,array}
\usepackage{pdflscape}
\usepackage{bm,upgreek}

\makenomenclature

\usepackage{epstopdf}
\usepackage{multirow}
\usepackage{float}
\usepackage[cmex10]{amsmath}
\usepackage {amssymb}
\usepackage{graphicx}
\usepackage{array}
\usepackage{mdwmath}
\usepackage{mdwtab}
\usepackage{eqparbox}
\usepackage{nomencl}
\usepackage{algorithm}
\usepackage{algorithmic}
\usepackage{makeidx}
\usepackage{ifthen}
\usepackage{subfigure}
\usepackage{caption}
\usepackage{amsthm}

\renewcommand{\nomgroup}[1]{%
\ifthenelse{\equal{#1}{I}}{\item[\textbf{Indices}]}{%
\ifthenelse{\equal{#1}{A}}{\item[\textbf{Abbreviations}]}{%
\ifthenelse{\equal{#1}{V}}{\item[\textbf{Variables}]}{%
\ifthenelse{\equal{#1}{P}}{\item[\textbf{Parameters and Constants}]}{%
}
}
}
}
}

\usepackage{multirow,tikz,enumitem}
\usepackage{xspace,amsmath,amsfonts,amssymb,tikz,multirow,float,newclude,pgfplots, stmaryrd,lipsum,mathtools, tikz, amsthm,changepage }
\usepackage{caption}
\usepackage{xspace}
\usepackage{latexsym}
\usepackage{xcolor}
\usepackage{graphicx}
\usepackage{enumitem}
\usepackage{comment,verbatim}
\usepackage{array}
\usepackage{booktabs}
\usepackage{fancyhdr}

\usepackage{cite}
\usepackage{comment}
\usepackage{epstopdf}
\usepackage{float}

\newcommand{\beq}{\begin{equation}}
\newcommand{\eeq}{\end{equation}}
\newcommand{\beqn}{\begin{eqnarray}}
\newcommand{\eeqn}{\end{eqnarray}}
\newcommand{\beqno}{\begin{eqnarray*}}
\newcommand{\eeqno}{\end{eqnarray*}}
\newcommand{\bma}{\begin{displaymath}}
\newcommand{\ema}{\end{displaymath}}
\newcommand{\bnu}{\begin{enumerate}}
\newcommand{\enu}{\end{enumerate}}
\newcommand{\bce}{\begin{center}}
\newcommand{\ece}{\end{center}}
\newcommand{\btb}{\begin{tabular}}
\newcommand{\etb}{\end{tabular}}

\newtheorem{theorem}{Theorem}[section]

\newtheorem{proposition}[theorem]{Proposition}

\def\ba{{\mathbf{a}}}
\def\by{{\mathbf{y}}}
\def\bx{{\mathbf{x}}}
\def\bz{{\mathbf{z}}}
\def\be{{\mathbf{e}}}

\def\bq{{\mathbf{q}}}
\def\bp{{\mathbf{p}}}
\def\bs{{\mathbf{s}}}
\def\blambda{{\mathbf{\lambda}}}

\allowdisplaybreaks[4]

\begin{document}



\title{Robust Dynamic Edge Service Placement Under Spatio-Temporal Correlated Demand Uncertainty}
\vspace{-0.65cm}
\author{\IEEEauthorblockN{Jiaming Cheng,~\IEEEmembership{Student Member,~IEEE}, Duong~Thuy~Anh~Nguyen,~\IEEEmembership{Student Member,~IEEE}, \\ and Duong Tung Nguyen,~\IEEEmembership{Member,~IEEE} } \vspace{-0.4cm}
\thanks{The authors are with the School of Electrical, Computer and Energy Engineering, Arizona State University, Tempe, AZ, USA. Email: \{jiaming, dtnguy52,duongnt\}@asu.edu.} \vspace{-0.4cm}
}

\maketitle

\begin{abstract}
Edge computing allows Service Providers (SPs) to enhance user experience by placing their services closer to the network edge. Determining the optimal provisioning of edge resources to meet the varying and uncertain demand cost-effectively is a critical task for SPs. This paper introduces a novel two-stage multi-period robust model for edge service placement and workload allocation, aiming to minimize the SP's operating costs while ensuring service quality. The salient feature of this model lies in its ability to enable SPs to utilize dynamic service placement and leverage spatio-temporal correlation in demand uncertainties to mitigate the inherent conservatism of robust solutions. In our model, resource reservation is optimized in the initial stage, preemptively, before the actual demand is disclosed, whereas dynamic service placement and workload allocation are determined in the subsequent stage, following the revelation of uncertainties. To address the challenges posed by integer recourse variables in the second stage of the resulting tri-level adjustable robust optimization problem, we propose a novel iterative, decomposition-based approach, ensuring finite convergence to an exact optimal solution. Extensive numerical results are provided to demonstrate the efficacy of the proposed model and approach.
\end{abstract}


\printnomenclature

\section{Introduction}

Edge Computing (EC) has emerged as a pivotal computing paradigm to improve user experiences and support a diverse range of Internet of Things (IoT) applications \cite{wshi16}. The increasing reliance on edge resources and the strict service requirements they must meet present substantial challenges for modern network operations. These challenges are further exacerbated by various system uncertainties, which can lead to suboptimal solutions when using deterministic optimization models if actual parameters diverge from their predicted values, ultimately affecting performance and system reliability. 

This paper investigates the edge service placement and resource management problem, focusing on the complexity introduced by 
the heterogeneity of geo-distributed edge nodes (ENs). Resource costs among ENs vary significantly due to factors such as hardware specifications, electricity costs, location, reliability, reputation, and ownership. Consequently, even though certain ENs may be closer to end-users, they are often bypassed for service placement because of their higher costs and resource prices, underscoring the complexities of strategic location selection. Also, efficient resource management and workload allocation are critical for service providers (SPs). Decisions regarding edge resource provisioning are often made amidst uncertainties about future demand. Neglecting these uncertainties can lead to over-provisioning or under-provisioning of resources, resulting in wastage and increased costs or degraded service quality and unmet demand, respectively.

To address the aforementioned challenges, we present a novel two-stage multi-period robust model tailored specifically for edge service placement and workload allocation, which aims to minimize the operating costs for SPs while ensuring high service quality. Service placement and resource procurement are fundamental for optimizing network performance and enhancing user satisfaction. Specifically, strategic service placement is essential for establishing proximity to users, significantly influencing user experience and service efficiency. Resource procurement, or sizing, decisions define the 
resource capacities to serve areas effectively. 
Our proposed model aims to optimize service placement and sizing decisions under \textit{spatio-temporal demand uncertainty},  minimizing costs while improving Quality of Service (QoS) in terms of 
latency and unmet demand. 


Numerous efforts in optimizing edge systems under uncertainty have emerged, primarily utilizing stochastic optimization (SO) and robust optimization (RO). The SO approach typically requires precise knowledge of the probability distribution of uncertainty, which can be challenging. 
Conversely, 
 RO adopts predefined uncertainty sets, 
enhancing computational tractability \cite{RO}. However, solutions derived from traditional 
RO models can be overly conservative, risking significant economic losses or suboptimal 
performance during actual operation. To address this, a two-stage RO framework was proposed in \cite{Duong_iot,resilientEC}, incorporating ``\textit{here-and-now}" decisions in the first stage, made before the uncertainty is revealed, and \textit{``wait-and-see}" recourse decisions in the second stage. 
This approach leverages newly revealed information to obtain less conservative 
solutions. 
However, most existing literature 
focuses on utilizing \textit{static} uncertainty sets (e.g., static polyhedral sets) which are predefined or specified with predetermined parameters to achieve desired robustness levels. Such an approach neglects the systematic representation and explicit consideration of temporal and spatial correlations.  
Incorporating these factors can significantly enhance the accuracy of prediction results \cite{yu2021short,zhang2021cloudlstm}.

This motivates us to explore an alternative method to accommodate inherent temporal and spatial dynamics of uncertainties,  
which can substantially reduce the size of the uncertainty set, thereby mitigating the conservatism in the solution. This reduction is possible due to the predictable relationships indicated by correlations between uncertain parameters. Also, dynamic uncertainty sets provide flexibility by allowing variations over time or in response to different scenarios. 
Capturing interdependencies and joint behavior among uncertainties allows decision-makers to gain insights into demand patterns, thereby alleviating the inherent conservatism in robust solutions. However, developing frameworks to model and integrate the intricate correlation structure of demand uncertainty remains a formidable research challenge and an crucial research area.


Another limitation of the existing 
robust edge service placement models is their assumption that service placement and resource procurement decisions are fixed and unalterable during the operational stage. This assumption stems from the model's design, which determines these decisions solely in the first stage before uncertainty is disclosed. \textit{To enhance the model's effectiveness, it is imperative to develop a more sophisticated framework that incorporates the dynamic nature of service placement decisions across different time slots.} 
It is neccessary to provide SPs with the flexibility to dynamically 
adjust service placement and resource allocation decisions to accommodate the time-varying resource demands from different areas.

To this end, this study formulates a novel two-stage multi-period dynamic service placement and sizing problem, 
where an SP strategically reserves resources before gaining knowledge of the actual demand. With the first-stage decisions established, the SP then determines service placement, resource adjustment, 
and resource allocation decisions during the operational stage.
Unlike our previous work \cite{Duong_iot,resilientEC}, which presented traditional two-stage robust models, this paper introduces a two-stage multi-period robust model capable of capturing the spatiotemporal correlation of uncertainties, considerably reducing costs for the SP. Furthermore, 
\textit{incorporating dynamic service placement 
introduces binary recourse variables in the second-stage problem, substantially elevating the optimization problem's complexity.} 
The non-convex nature of the second-stage problem, due to the presence of binary variables, renders traditional decomposition-based methods, which iteratively generate variables and constraints to 
refine the solution space, such as those in \cite{CCGARO,Duong_iot,resilientEC}, ineffective.
To bridge this gap, we develop a novel decomposition algorithm that optimally solves this trilevel model. 
Our contributions are threefold: 



\begin{enumerate}
    \item \textbf{\textit{Modeling:}} We propose a novel two-stage, multi-period robust model that optimizes resource reservation in the first stage and facilitates dynamic edge service placement and resource adjustments during operation. Furthermore, we introduce a data-driven approach to model dynamic uncertainty sets by integrating temporal and spatial dynamics of demand fluctuations. This method inherently offers a superior trade-off between robustness and profitability for SPs, emphasizing the impact and advantages of dynamic service placement decisions on network performance.
    \item \textbf{\textit{Solution Approach:}} The presence of integer recourse variables in the second stage poses significant challenges, limiting the direct application of existing algorithms \cite{CCGARO,Duong_iot,resilientEC}. To overcome this, we develop a novel iterative, decomposition-based approach, called \textit{ROD}, specifically designed for solving the resulting Adjustable Robust Optimization (ARO) problem with integer recourse. This method efficiently refines robust solutions, supported by mathematical proofs confirming its optimality and demonstrating convergence within a finite number of iterations.
    \item \textbf{\textit{Simulation:}} We conduct extensive numerical simulations to demonstrate the superior performance of our proposed model over benchmarks and previous models presented in \cite{Duong_iot}. Sensitivity analyses explore the impacts of critical system parameters on the optimal solution. Run-time analysis on large-scale systems validates the effectiveness of the proposed algorithm in practical scenarios.
\end{enumerate}
The rest of this paper is organized as follows: Section \ref{related_work} discusses the related work. Section \ref{modelfor} presents the system model and problem formulation. In Section \ref{solution}, we describe the proposed solution approach. Simulation results are presented in Section \ref{Numerical_results}, followed by the conclusions in Section \ref{summary}.

\section{Related work}
\label{related_work}
Recent works have explored the 
edge service placement problem from diverse perspectives. Mechtri \textit{el at.} \cite{static16} proposed a heuristic algorithm for optimal edge server selection to maximize revenue while meeting service requirements. Yang \textit{el at.} \cite{Yang2018} focused on service placement and resource allocation in edge cloud networks, presenting an incremental allocation algorithm to minimize operational costs and ensure low-latency requirements. Reference \cite{Lcai20} addressed time-scale computing resource management for edge servers, addressing wholesale and buyback schemes, and cloud-related pricing strategies. 
In \cite{Niyato20}, a dynamic service placement model with reliability awareness was introduced to minimize total operating costs while maximizing the number of admitted services.
Zhang \textit{el at.} \cite{zhang13} presented a model predictive control-based algorithm for multi-period dynamic service placement to minimize operating costs.  A multiple edge service placement problem was proposed in \cite{He19} to maximize the total reward. \textit{However, these works have predominantly concentrated on deterministic formulations, without considering system uncertainties.} 

There is an expanding literature on edge service placement under uncertainty. Ouyang \textit{el at.} \cite{Ouyang18} conceptualized the dynamic service placement problem as a contextual multi-armed bandit problem and proposed a Thompson-sampling-based online learning algorithm. This algorithm selects edge nodes for task offloading, minimizing total delay and service migration costs. Cheng \textit{el at.} \cite{Cheng_bandit} explored a bandit-based online posted pricing mechanism, utilizing a bandit learning algorithm to maximize the profit of the edge-cloud platform under demand uncertainty. Reference \cite{jia18} investigated the dynamic placement of virtualized network function (VNF) service chains across geographically distributed cloudas. The aim was to serve flows between dispersed source and destination pairs while minimizing the total operating cost across the entire system. \textit{ These studies leverage online algorithms for dynamic service placement to optimize system performance under uncertainty.}


Another prevalent method for addressing uncertainty is SO. Reference \cite{niyato12} introduced a stochastic model for cloud resource provisioning to minimize the total provisioning cost. 
Mireslami \textit{el at.} \cite{Far21} formulated a stochastic cloud resource allocation problem, incorporating dynamic user demand to minimize the total deployment cost. 
Reference \cite{Grosu20} casts the energy-aware service placement problem in edge-cloud networks as a multi-stage stochastic program that aims to maximize the QoS while respecting the limited energy budget of edge servers.

Recently, RO has gained significant attention 
as an alternative methodology for optimizing edge systems under uncertainty. 
The robust Service Function Chain (SFC) placement problem under demand uncertainty was formulated as a quadratic integer program in \cite{jli21} to maximize the total profit. Nguyen \textit{el at.} \cite{nguyen20} employed RO to formulate a deadline-aware co-located and geo-distributed SFC orchestration with demand uncertainty and developed both exact and approximate algorithms for solving the problem. Reference \cite{Duong_iot}  introduced a two-stage robust model to help the SP optimize edge service placement and sizing. In \cite{resilientEC}, the authors formulated a resilience-aware edge service placement and workload allocation, considering both resource demand and node failure uncertainties. Chen \textit{el at.} \cite{yche21} addressed a distributionally robust computation-intensive task offloading problem that aims to minimize the expected latency under the worst-case probability distribution. Wei \textit{el at.} \cite{fwei23} introduced a hybrid data-driven framework leveraging prediction intervals and robust optimization to reduce cost-waste during resource provisioning. Cheng \textit{et al.} \cite{cheng_ddu} proposed RO models with a decision-dependent uncertainty set to improve user experience considering uncertain link delay.

However, \textit{the existing literature has overlooked the issues of spatiotemporal demand correlation and dynamic placement for edge service provisioning}. Our work addresses these gaps by proposing a novel robust model and solution approaches. 

\section{System Model and Problem Formulation}
\label{modelfor}
 
\subsection{System Model}
\label{model}

\begin{table}[t!]
\newcommand{\tabincell}[2]{\begin{tabular}{@{}#1@{}}#2\end{tabular}}  
\centering
\caption{Notation}
\begin{tabular}{|c|l|}
\hline
\textbf{Notation} & \qquad\qquad\qquad\quad \textbf{Meaning} \\ \hline
\multicolumn{2}{|c|}{\textbf{Sets and indices}}                                       
\\ 
\hline
\!\!\!$t$, $T$, $\mathcal{T}$ & \!\! Indices, number and set of time periods\\
\hline
\!\!\!$j$,$J$,$\mathcal{J}$ & \!\! Indices, number, and set of ENs\\ 
\hline
\!\!\!$i$,$I$,$\mathcal{I}$ & \!\! Indices, number, and set of access points\\ 
\hline
\!\!\!$m$ & Index of source EN for service downloading \\
\hline
\!\!\!$\mathcal{D}_1$ & \!\! Static uncertainty set (SUS)\\
\hline
\!\!\!$\mathcal{D}_2$ & \!\! Dynamic uncertainty set (DUS)\\
\hline
\multicolumn{2}{|c|}{\textbf{Parameters}}                                       \\
\hline
\!\!\! $d_{i,j}(d_{i,0})$ \!\!\! & \!\!  Network delay between AP $i$ and EN $j$ (cloud)\\
\hline
\!\!\!$\rho$ \!\!\! & \!\!  Delay penalty parameter  \\
\hline
\!\!\!$\delta$ \!\!\! & \!\!  Length of a time slot  \\
\hline
\!\!\!$H_{i,j} (H_{i,0})$ \!\!\!\!\!\! & \!\! Number of network hops between AP $i$ and EN $j$ (cloud)\\
\hline
\!\!\!$b$ \!\!\! & \!\!  Data size of one request\\
\hline
\!\!\!$C_{j}$   & \!\! Resource capacity of EN $j$\\
\hline
\!\!\!$h_{m,j} (h_{0,j})$ \!\!\!\!\!\! & \!\! Service download cost from EN $m$ to EN $j$\\
\hline
\!\!\!$f_{j}^{t}$ & \!\! Service placement cost at EN $j$ at time $t$\\
\hline
\!\!\!$f_{j,t}^{\sf s}$ & \!\! Storage cost at EN $j$ at time $t$\\
\hline
\!\!\!$p_{j}^{t} (p_{0}^{t})$ & \!\! Reserved resource unit price at EN $j$ (cloud) at time $t$ \!\!\!\!\!\!\\
\hline
\!\!\!$e_{j}^{t} (e_{0}^{t})$ & \!\! ``\textit{Buy-more}" resource unit price at EN $j$ (cloud) at time $t$ \!\!\!\!\!\\
\hline
\!\!\!$a_{j}^{t}$ & \!\! ``\textit{Sell-back}" resource unit price at EN $j$ (cloud) at time $t$\!\!\!\!\\
\hline
\!\!\!$w$   & \!\! Average resource demand per  request \\
\hline
\!\!\!$\beta$ & \!\! Unit price of network bandwidth\\
\hline
\!\!\!$\lambda_{i}^{t}$ & \!\! Demand at AP $i$ at time $t$\\
\hline
\!\!\!$\varsigma_{j}^{t}$ & \!\! Auxiliary parameter: $f_j^{t} (1 - z_j^{t-1}) + f_{j,t}^{\sf s}$ \\
\hline
\!\!\!$c_{i,0}^{t}$ & \!\! Auxiliary parameter: $\rho d_{i,0} + \beta b H_{i,0}$\\
\hline
\!\!\!$c_{i,j}^{t}$ & \!\! Auxiliary parameter $\rho d_{i,j} + \beta b H_{i,j}$\\
\hline
\multicolumn{2}{|c|}{\textbf{Variables}} \\ 
\hline
\!\!\!$s_{j}^{t} (s_{0}^{t})$ \!\!\! & \!\! \tabincell{l}{Amount of reserved resource at EN $j$ (cloud) at time $t$} \\
\hline
\!\!\!$z_{j}^{t}$ & \!\! $\{0,1\}$, ``$1$" if the service is placed at EN $j$ at time $t$\\
\hline
\!\!\! $q_{m,j}^{t} (q_{0,j}^{t})$ \!\!\! & \!\!   \tabincell{l}{$\{0,1\}$, ``$1$" if the service is downloaded from EN $m$ \\ (cloud) to EN $j$ at time $t$ and ``$0$", otherwise} \\
\hline
\!\!\! $y_j^{B,t} (y_0^{B,t})$ \!\!\!\! & \!\!  Amount of ``\textit{buy-more}''  resources at EN $j$ (cloud) at time $t$ \!\!\!\!\!\!\! \\
\hline
\!\!\! $y_j^{S,t} (y_0^{S,t})$ \!\!\! & \!\!  Amount of  ``\textit{sell-back}''  resources at EN $j$ (cloud) at time $t$ \!\!\! \!\!\\
\hline
\!\!\! $x_{i,j}^{t} (x_{i,0}^{t})$ \!\!\! & \!\!  Workload allocated to EN $j$ (cloud) from AP $i$ \\
\hline 
\end{tabular} \label{notation}
\end{table}

This paper addresses the joint service placement and resource procurement problem for an SP serving users across different areas, each represented by an access point (AP). The set of APs is denoted by $\mathcal{I}$, with a total of $I$ APs. To enhance service quality, the SP can procure edge resources from an EC platform that manages a set  $\mathcal{J}$ of $J$ geographically distributed, heterogeneous ENs. Let $i$ and $j$ signify the indices for APs and ENs, respectively.
Traditionally,  in the absence of EC, user requests are routed to remote cloud servers for processing. Leveraging EC allows these requests to be directly served at the edge,
effectively mitigating network delays between users and computing nodes. 
The network delay between AP $i$ and EN $j$ is represented as $d_{i,j}$, while the network delay between AP $i$ and the cloud is denoted bywe $d_{i,0}$.

A critical task for the SP is to jointly optimize 
edge resource procurement, service placement, and 
workload allocation. 
This is a complex task due to the heterogeneity of the ENs concerning resource capacity, geographical location, and pricing. We adopt a time-slotted model with $T$ equal-length periods, where $t$ denotes the time index and $\mathcal{T}$ for the set of time periods. Before the operational stage, the SP must ascertain the quantity of reserved resources to commit to procurement. The unit prices of reserved resources at the cloud and EN $j$ at time $t$ are $p_0^t$ and $p_j^t$, respectively. Let $s_0^t$ and $s_j^t$ denote the amounts of reserved resources at the cloud and EN $j$, respectively, at time $t$.

Due to the uncertain and time-varying demand, the committed resource reservation may significantly differ from the actual demand.
We assume that in the operational stage, the SP can adjust computing resources
based on actual demand.
Specifically, the SP has the flexibility to procure additional resources in the event of a deficiency and is also allowed to sell surplus resources back to the platform. To prevent arbitrage, the platform can impose constraints ensuring that the  ``\textit{sell-back}" price is lower than the reserve price and the ``\textit{buy-more}" price is higher than the reserve price \cite{Lcai20}. Define $a_o^t$ and $a_j^t$ as the ``\textit{sell-back}" prices at the cloud and EN $j$ at time $t$, respectively. Also, $e_0^t$ and $e_j^t$ denote the corresponding ``\textit{buy-more}" prices at the cloud and EN $j$ at time $t$. We use $y_0^{B,t}$ and $y_j^{B,t}$ to indicate the amount of ``\textit{buy-more}" resources at the cloud and EN $j$ at time $t$, while $y_0^{S,t}$ and $y_j^{S,t}$ represent the amounts of ``\textit{sell-back}" resources at the cloud and EN $j$ at time $t$. 
Denote the resource trading vector by $\by = (y^{B},y^{S}) = (y_{j,t}^{B},y_{0,t}^{B},y_{j,t}^{S},y_{0,t}^{S})$, where superscripts $B$ and $S$ represent ``\textit{buy-more}" and ``\textit{sell-back}" decisions, respectively. While the model focuses on a single resource type for simplicity, it can be easily extended to include multiple resource types.
The dynamic resource trading model is illustrated in Fig.~\ref{fig: multi_period_model}.


\begin{figure}[h!]
\vspace{-0.3cm}
\centering
\includegraphics[width=0.4\textwidth,height=0.1\textheight]{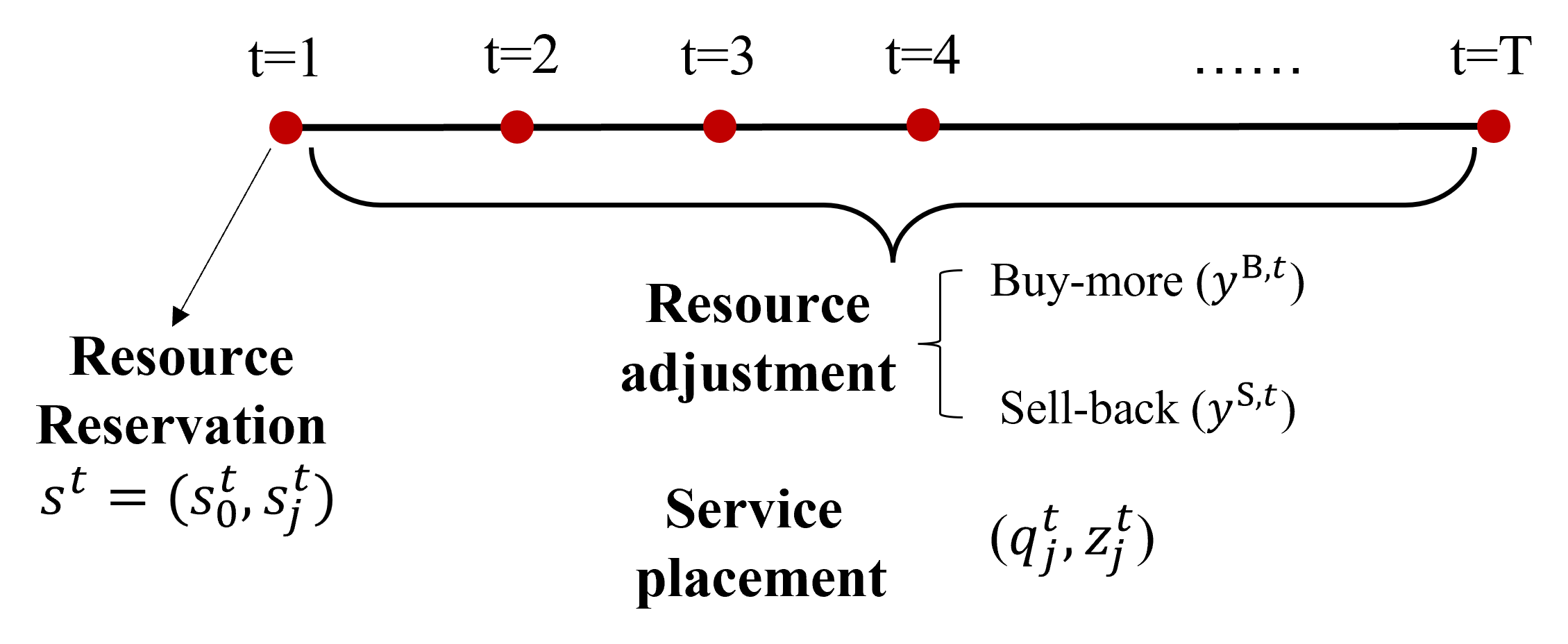}
\caption{Resource trading model}
\label{fig: multi_period_model}
\vspace{-0.3cm}
\end{figure}

Our proposed model aims to assist the SP in optimizing resource reservation under uncertainty. 
Besides resource trading, the SP must optimize service placement, assuming the service is consistently available in the cloud. Let the binary parameter $z_j^{0}$ signify the initial service placement status at EN $j$. 
We introduce the binary 
variable $z_j^{t}$, which equals 
one when the SP places the service onto EN $j$ at time $t$. 
This allows the SP to adapt its decisions based on actual demand realization. 
We also define the binary variable $q_{m,j}^{t}$ to indicate 
whether the service is downloaded from the source EN $m$ to the destination EN $j$ at time $t$. 
The binary variable $q_{0,j}^{t}$ indicates if 
the service is downloaded from the cloud to EN $j$.


\begin{figure}[t!]
\centering
\includegraphics[width=0.46\textwidth,height=0.2\textheight]{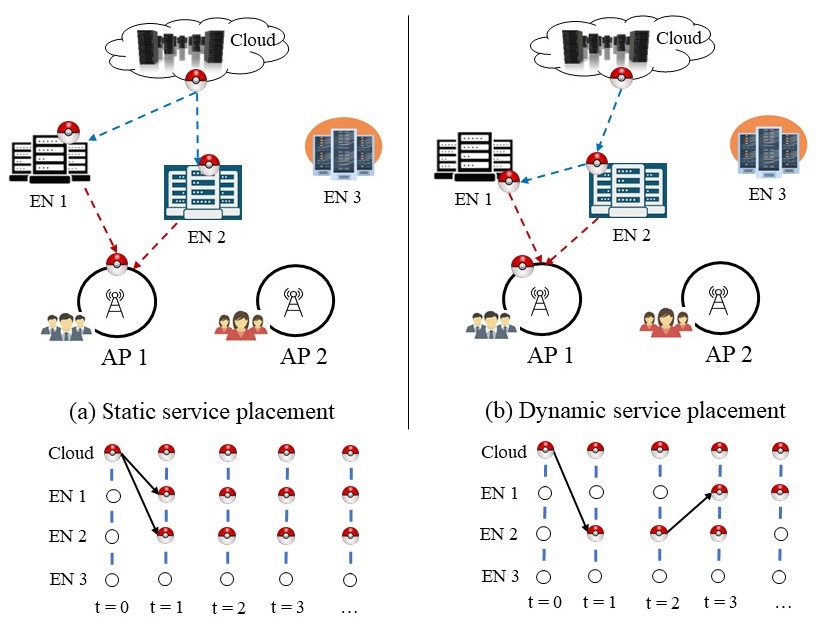}
\caption{System model}
\label{fig: model}
\vspace{-0.5cm}
\end{figure}

Fig.~\ref{fig: model} presents a simplified EC network with 2 APs, 3 ENs, and a  cloud data center (DC), comparing the traditional static service placement model 
with the dynamic service placement model. 
In the static model, service placement decisions 
remains fixed during operation, irrespective of demand fluctuations. In contrast, the dynamic service placement scheme 
allows flexible adjustments. 
This dynamic approach 
leverages nearby ENs for service downloading and installation, 
resulting in significant cost savings. It also enables the SP to optimize resource utilization, adapting to the spatio-temporally varying demands.

Besides service placement and resource trading, the SP must optimize workload allocation from different areas to the ENs and the cloud. Ideally, user demand in area $i$, denoted by $\lambda_i$, should be served by the nearest EN to minimize network latency. However, due to the limited capacity of each EN and the procured resources, strategic workload allocation is essential to ensure high QoS while minimizing costs. Given the uncertain demand at the initial stage, a portion of the workload may be directed to the remote cloud, potentially resulting in lower QoS. Let $x_{i,j}^{t}$ and $x_{i,0}^{t}$ denote the workload allocated from AP $i$ to EN $j$ and the cloud at time $t$, respectively. 



The following first presents a deterministic model tailored for the underlying problem. 
Then, we propose two distinct models to account for demand uncertainties. The first model employs a conventional polyhedral uncertainty set, while the second model incorporates spatial-temporal dynamics, encompassing changes over time and variations across different geographical locations.

 \vspace{-0.3cm}
\subsection{Deterministic Formulation}
The objective of the SP is to minimize the overall cost 
while concurrently improving the user experience. We will commence by introducing the cost model 
along with the associated constraints in the deterministic problem formulation.


\noindent
\textbf{\textit{1) Resource reservation cost:}} The cost of computing resources varies across ENs. 
Recall that $p_{0}^{t}$ and $p_{j}^{t}$ represent the reserved price for one unit of computing resources purchased from the remote cloud and EN $j$ at time $t$, respectively. 
Let $\delta$ denote the duration of a time slot. The total resource reservation cost is:
\begin{align}
\label{Cost:reserve}
    C^{\sf 1} =\! \sum_{t \in \mathcal{T}} \sum_{j \in \mathcal{J}} \delta p_j^{t} s_{j}^{t} +  \sum_{t \in \mathcal{T}} \delta p_{0}^{t} s_{0}^{t}. 
\end{align}

\noindent
\textbf{\textit{2) Resource adjustment cost:}} During the operational stage, the SP has the flexibility to adjust the reserved computing resources from the planning stage. In practice, the primary factor in determining if computing resources are surplus or insufficient is the relation between the reserved computing resources and the actual demand. This relationship will be further discussed in the (\ref{constr-det_adjust1})-(\ref{constr-det_adjust2}). Let $w$ represent the average amount of resources needed to process a request. If the reserved resources ($s_j^{t}$) are sufficient to meet demand (i.e., $s_j^{t} \geq w \sum_{i} x_{i,j}^{t}$), those surplus computing resources can be sold back to the market at a reduced price $\ba^{t} = (a_j^{t},a_0^{t})$, lower than the reserved price $\bp^{t}$. Conversely, if reserved resources are insufficient to meet demand, the SP can purchase additional computing resources $(y_{j,t}^{S})$ at a higher cost $\be^{t} = (e_j^{t},e^{t}_{0})$ compared to $\bp^{t}$. Maintaining the price relationship $\be^{t} \geq \bp^{t} \geq \ba^{t}$ is crucial to prevent any potential arbitrage. The resource adjustment cost can be expressed as:\vspace{-0.1cm}
\begin{align}
\label{Cost:adjus4t}
    C^{\sf 2} &=\! \sum_{t \in \mathcal{T}} \sum_{j \in \mathcal{J}} \delta \bigg[ e_j^{t} y_{j}^{B,t} \!-\! a_{j}^{t} y_{j}^{S,t} \!+\! e_{0}^{t} y_{0}^{B,t} \!-\! a_{0}^{t} y_{0}^{S,t} \bigg].
\end{align}

\noindent
\textbf{\textit{3) Service installation cost:}} The service installation cost (e.g., license cost, costs for setting up the service on an EN), denoted by $f_j$, is incurred when the service is placed onto EN $j$ that does not have the service. We assume that all services are installed in the cloud. Thus, the total service installation cost is: 
\vspace{-0.1cm}\begin{align}
\label{Cost:placement}
    C^{\sf 3} = \sum_{t \in \mathcal{T}} \sum_{j \in \mathcal{J}}f_j^{t} (1 - z_{j}^{t - 1}) z_{j}^{t}.
\end{align}

\noindent
\textbf{\textit{4) Service download cost:}} The SP can dynamically decide to download services from nearby ENs where services are already installed, rather than relying solely on the remote cloud. The cost of downloading a service from source EN $m$ to EN $j$ at time $t$ is denoted by $h_{m,j}^{t}$, and the cost of downloading services from the cloud to EN $j$ is represented by $h_{0,j}^{t}$. 
The total service download cost is expressed as follows: \vspace{-0.1cm}
\begin{align}
\label{Cost:download}
    C^{\sf 4} \!=\! \sum_{t \in \mathcal{T}} \bigg[ \sum_{j \in \mathcal{J}} \sum_{m \in \mathcal{J} \setminus \{j\} } h_{m,j}^{t} q_{m,j}^{t} + h_{0,j}^{t} q_{0,j}^{t} \bigg].
\end{align}

\noindent
\textbf{\textit{5) Storage cost:}} When the service is placed at EN $j$, it occupies the storage space of the EN. Let $f_j^{\sf s}$ be the storage cost at EN $j$. Then, the total storage cost is given as follows:
\begin{align}
\label{Cost:storage}
    C^{\sf 5} =\! \sum_{t \in \mathcal{T}} \sum_{j \in \mathcal{J}}  f_{j,t}^{s} z_{j}^{t}.
\end{align}


\noindent
\textbf{\textit{6) Network delay cost:}}
Service quality is evaluated based on network delay, representing the time taken for data to traverse between different areas. The cost associated with network delay is directly related to the amount of workload transferred over the network. The delay penalty, denoted by $\rho$, reflects the SP's preference for minimizing network delay. A higher value of $\rho$ indicates the SP's willingness to incur a greater cost for achieving lower network delay. The total network delay cost, encompassing both cloud and edge network delay, is given by:
\begin{align}
\label{Cost:delay}
    C^{\sf 7} = \rho \: \sum_{t \in \mathcal{T}} \bigg[ \sum_{i \in \mathcal{I}} d_{i,0} x_{i,0}^{t} + \sum_{i \in \mathcal{I}}\sum_{j \in \mathcal{J}} d_{i,j}x_{i,j}^{t} \bigg].
\end{align}


\noindent
\textbf{\textit{7) Network bandwidth cost:}} The bandwidth cost between two areas depends on the number of network hops and the data traffic between them. Let $H_{i,j}$ represent the number of network hops between AP $i$ and EN $j$, while $H_{i,0}$ denotes the number of hops between AP $i$ and the cloud. Given a unit price $\beta$ for network bandwidth 
and the average data size per request as $b$, the overall bandwidth cost can be calculated as follows:
\begin{align}
\label{Cost:bandwidth}
    C^{\sf 8} =  \beta b \sum_{t \in \mathcal{T}} \bigg[ \sum_{i \in \mathcal{I}} H_{i,0} x_{i,0}^t + \sum_{i \in \mathcal{I}} \sum_{j \in \mathcal{J}} H_{i,j} x_{i,j}^{t} \bigg].
\end{align}

Overall, the deterministic model can be formulated as: 
\begin{subequations}
\label{DET}
\begin{align}
& \!\! \textbf{DET} \!\!: \underset{\bs,\bx,\bq,\bz,\by}{\text{min}} ~ \mathcal{C}^1 \!+\! \mathcal{C}^2 \!+\! \mathcal{C}^3 \!+\! \mathcal{C}^4 + \mathcal{C}^5 \!+\! \mathcal{C}^6 \!+\! \mathcal{C}^7 
\label{eq-det_obj}\\
& \quad \quad  \quad ~ \text{s.t.} \quad (\ref{det_cap}) - (\ref{det_decision}).
\end{align}
\end{subequations}
To simplify the notation, we introduce auxiliary cost parameters $\varsigma_j^{t}= f_j^{t} (1 - z_j^{t-1}) + f_{j,t}^{\sf s}, \forall j,t$, encapsulating the costs associated with service placement for each EN. Additionally, we define the cost parameters $c_{i,0} = \rho d_{i,0} + \beta b H_{i,0}$ and $c_{i,j} = \rho d_{i,j} + \beta b H_{i,j}$.
We will describe constraints (\ref{det_cap})--(\ref{det_decision}) in the following. 

\noindent
\textbf{\textit{Reserved resource capacity constraints:}} 
The total reserved resources at EN $j$ cannot exceed its capacity $C_j$.
For simplicity, we assume the cloud has unlimited resources. \vspace{-0.2cm}
\begin{align}
    s_{0}^{t} \geq 0, \forall t, ~ 0 \leq s_{j}^{t} \leq C_j,~ \forall j,t.\label{det_cap}
\end{align}

\noindent
\textbf{\textit{Workload allocation constraints:}} The demand from each area must be processed either at ENs or at the cloud:\vspace{-0.2cm}
\begin{align}
    x_{i,0}^{t} + \sum_j x_{i,j}^{t} \geq \lambda_{i}^{t}, ~ \forall i,t. \label{det_supply}
\end{align}

\noindent
\textbf{\textit{Resource procurement constraints:}} Similar to (\ref{det_cap}), the total procured resources at an EN is limited by its capacity. Furthermore, the amount of resources sold back to the market should not exceed the reserved resources in the planning stage, as specified by constraints (\ref{det_sellback1}) and (\ref{det_sellback2}). 
\begin{subequations}
\label{constr-det_adjust1}
\begin{align}
    s_{j}^{t} + y^{B}_{j,t} - y^{S}_{j,t} \leq C_j z_{j}^{t} \leq C_j^t z_{j}^{t}, ~ \forall j,t \label{det_adjust_cap} \\
    0 \leq y^{S}_{j,t} \leq s_{j}^{t}, ~ \forall j,t \label{det_sellback1}\\
    0 \leq y^{S}_{0,t} \leq s_{0}^{t}, ~ \forall t. \label{det_sellback2} 
\end{align}
\end{subequations}
Let $w$ represent the average amount of resources needed to process a request.The following constraints ensure the procured resources are sufficient to handle the allocated workload: 
\begin{subequations}
\label{constr-det_adjust2}
\begin{align}
    s_{0}^{t} + y^{B}_{0,t} - y^{S}_{0,t} \geq w \sum_i x_{i,0}^{t}, ~ \forall t \label{det_adjust1}\\
    s_{j}^{t} + y^{B}_{j,t} - y^{S}_{j,t} \geq  w \sum_i x_{i,j}^{t}, ~ \forall j,t. \label{det_adjust2}
\end{align}
\end{subequations}

\noindent
\textbf{\textit{Service placement constraints:}} The service can be downloaded from either the cloud or nearby ENs that have installed the service. Constraints (\ref{det_place1}) imply that the SP cannot download the service from EN $m$ to any other ENs at time $t$ if the service is not placed at EN $m$ at time $t-1$. 
Besides, if the service is not installed at EN $j$ at time $t-1$ and is installed at EN $j$ at time $t$, it has to be downloaded from either the cloud or nearby ENs, shown in constraints (\ref{det_place2}). Our model can be extended to incorporate other placement constraints, such as the minimum active time required for each service placement decision.
\vspace{-0.15cm}
\begin{subequations}
\label{constr-det_place}
\begin{align}
    \sum_{j}  q_{m,j}^{t} \leq z_{m}^{t-1}, \: \forall m,t \label{det_place1}\\
    \sum_{m} q_{m,j}^{t} + q_{0,j}^{t} \geq z_{j}^{t} - z_{j}^{t-1}, ~ \forall j,t.  \label{det_place2}
\end{align}
\end{subequations}

\noindent
\textbf{\textit{Constraints on decision variables:}} \vspace{-0.2cm}
\begin{subequations}
\label{det_decision}
\begin{align}
    x_{i,0}^{t} \geq 0, ~ \forall i,t ;~ x_{i,j}^{t} \geq 0, ~ \forall i,j,t   \label{det_var1} \\
    z_j^{t}, q_{0,j}^{t} \in \{0,1\},  ~ \forall j,t ; ~ q_{m,j}^{t} \in \{0,1\}, ~ \forall m,j,t.  \label{det_var2} 
\end{align}
\end{subequations}

\subsection{Uncertainty Modeling}
\label{uncertainty}


In the deterministic model (\textbf{DET}), the demand $\lambda_i$ during the scheduling horizon is assumed to be precisely known.  
However, demand often fluctuates and is uncertain. It can significantly differ from the forecast value. Thus, the \textbf{DET} model can result in suboptimal solutions. 
We propose two approaches for modeling demand uncertainty. The first and conventional approach models static demand uncertainty, treating demands across different time slots as independent variables without any correlation. In contrast, the second approach models correlated demand uncertainty, capturing the spatial-temporal dynamics of demand. This model considers dependencies and trends in demand across different times and locations, enabling more informed and effective decision-making.

\subsubsection{Static uncertainty set (\textit{SUS})}
We leverage the traditional RO approach to model the demand uncertainty. The actual demand aggregated at AP $i$ is considered to fluctuate within the range of $[\bar{\lambda}_{i}^{t} - \Tilde{\lambda}_{i}^{t}, \bar{\lambda}_{i}^{t} + \Tilde{\lambda}_{i}^{t}]$. Here, $\bar{\lambda}_{i}^{t}$ represents the forecast demand for AP $i$ at time $t$, while $\Tilde{\lambda}_{i}^{t}$ captures the maximum  demand variation. To represent the uncertain demand, we define a polyhedral uncertainty set as follows \cite{Duong_iot}:
\begin{subequations}
\label{static_uncertainty}
\begin{align}
    \mathcal{D}_1(\lambda) = \bigg\{& \lambda_{i}^{t}: \lambda_{i}^{t} = 
    \bar{\lambda}^{t}_{i} + g_{i}^{t} \Tilde{\lambda}_{i}^{t}, \: g_{i}^{t} \in [-1,1], \forall i,t, \\
     & \sum_{i \in \mathcal{I}} |g_{i}^{t}| \leq \Gamma_1, ~ \forall t \bigg\},
\end{align}
\end{subequations}
where the parameter $\Gamma_1$ controls the robustness of the optimal solution. Specifically, a larger value of $\Gamma_1$ indicates higher conservatism in the optimal solution. 

\subsubsection{Dynamic uncertainty set (\textit{DUS})}
\label{dynamic_uncertainty}
The traditional static uncertainty set $\mathcal{D}_1 (\lambda)$ fails to capture the dynamic demand correlation. To explicitly model the spatio-temporal correlations between demands across different areas over time, we propose a multi-period dynamic demand uncertainty set that utilizes a linear dynamic system. Fig.~\ref{fig: static_dynamic} visualizes the main difference between the static and dynamic uncertainty sets. The red dots show all possible uncertain realizations, which are captured by the blue contour, i.e., shape of the predefined uncertainty set. Since our problem considers the worst-case demand (e.g., largest demand) within the uncertainty set. The green line/dots indicates the extreme value within the uncertainty set. Clearly, it shows that the DUS is less conservative compared to SUS since SUS remain the same with $t$. The dynamic uncertainty set $\mathcal{D}_2(\lambda)$ is defined as follows:
\vspace{-0.2cm}
\begin{subequations}
\label{dyanmic_uncertainty}
\begin{align}
    \mathcal{D}_2(\lambda^{t-L:t-1}) = \bigg\{ & \lambda_{i}^{t}: ~ \lambda_{i}^{t} = \bar{\lambda}_{i}^{t} + \Tilde{\lambda}_{i}^{t}, \forall i,t \label{spatialtemperal} \\
    & \Tilde{\lambda}_{i}^{t} = \sum_{s = 1}^{L} A^{s}_{i} \Tilde{\lambda}_{i}^{t - s} +  \underbrace{\mathbf{B}_{i}  g_{i}^{t}}_{\epsilon_{i}^{t}}, ~ \forall i,t \label{deviation}\\
    \sum_{i \in \mathcal{I}}  &|g_{i}^{t}| \leq \Gamma_1, ~\forall t;  ~ g_{i}^{t} \in [-1, 1], \forall i,t \label{box} \bigg\}.
\end{align}
\end{subequations}
Equation (\ref{spatialtemperal}) represents the actual demand $\lambda_{i}^{t}$ as a combination of the forecast demand $\bar{\lambda}_i^{t}$ and a deviation $\Tilde{\lambda}_i^{t}$. The forecast demand can be estimated using historical data. The deviation  $\Tilde{\lambda}_i^{t}$ incorporates a linear dynamic relationship involving residuals ${\Tilde{\lambda}_i^{t-s}:s=1,\ldots,L}$, observed in previous time slots from $t-L$ to $t-1$, along with an error term $\epsilon_{i}^{t}$. Here, the parameter $L$ serves as a time lag operator. In the \textit{DUS} framework, the deviation $\Tilde{\lambda}_{i}^{t}$ follows a multivariate auto-regressive process of order $L$, determined by the innovation process ${\epsilon_{i}^{t}}$. The parameter $A_i^s$ captures the temporal correlation between $\Tilde{\lambda}_i^{t}$ and $\Tilde{\lambda}_i^{t-s}$, while $\mathbf{B}_{i}$ quantifies the spatial relationship of traffic demand at adjacent areas at time $t$. 
\begin{figure}[h!]
\vspace{-0.2cm}
\centering
\includegraphics[width=0.44\textwidth,height=0.115\textheight]{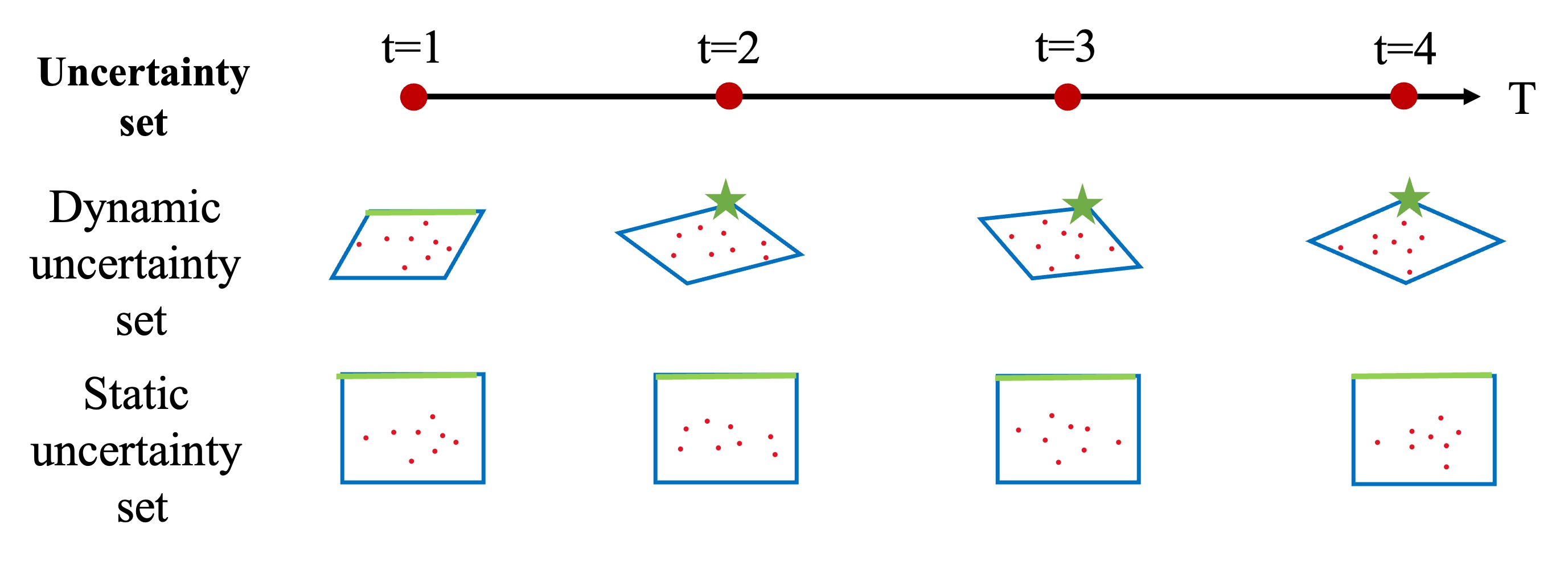}
\caption{Uncertainty comparison}
\label{fig: static_dynamic}
\end{figure}
\vspace{-0.3cm}

The error terms in the model, represented by the vector $[\epsilon_{1}^{t},\ldots,\epsilon_{I}^{t}]^\mathsf{T}$, are independent across different time slots and follow a normal distribution with mean $0$ and covariance matrix $\Sigma$. Similar to \textit{SUS}, $\Gamma_1$ controls the size of uncertainty and influences the robustness of the optimal solution, while the parameter $g_{i}^{t}$ governs the maximum forecast error relative to the forecast demand. The objective of this linear dynamic model is to dynamically incorporate demand while effectively leveraging all available statistical information derived from historical data. In Section \ref{time_series_process}, 
we will discuss the procedure to estimate those parameters based on multivariate auto-regressive tools.

Fig. \ref{fig: static_dynamic} compares DUS and SUS, highlighting the key differences between them. The red dots represent historical data, with uncertainty derived from all possible realizations. In RO, we aim to encompass all data points within a specific parametric uncertainty set, represented by the blue contour. 
The shapes of these uncertainty sets, such as box, polyhedron, or ellipse, reflect different assumptions about data uncertainties. Unlike SUS, which remains constant over time as it treats demand at each interval independently, DUS adapts over time, influenced by the spatial-temporal dynamics of demand variability, allowing for a more responsive representation of uncertainty and resulting in a less conservative robust solution.

\subsection{Adaptive Robust Optimization Problem Formulation}
\label{prob_form}
The proposed robust framework consists of two stages. In the planning stage, the SP reserves computing resources at ENs and the cloud at reserved prices. Given the first-stage decisions $\bs$ and the worst-case demand realization $\lambda$, the SP proceeds to make real-time operational decisions, including service placement $(\bz)$, service download $(\bq)$, resource adjustment $(\by)$, and workload allocation $(\bx)$ decisions. These definitions allow us to represent the costs associated with delay and network hops in the subsequent formulation of the ARO framework:
\begin{subequations}
\label{aro_model}
\begin{align}
& \!\!\mathcal{P}_1 \!\!: \underset{\bs\in \Omega_{1}(s)}{\text{min}} ~ \mathcal{C}^1  + \underset{\lambda \in \mathcal{D}_{2}}{\max} \min_{\substack{\bx,\bq,\bz,\by\\ \in \Omega_{2}(\mathbf{s},\mathbf{\lambda})}} \Big\{ \mathcal{C}^2 + \mathcal{C}^3 + \mathcal{C}^4 \nonumber \\
&  \qquad\qquad\qquad\quad \quad \quad \quad \quad \quad + \mathcal{C}^5 +  \mathcal{C}^6 + \mathcal{C}^7 + \mathcal{C}^{8} \Big\} 
\label{aro_obj}\\
& ~~ \text{s.t.} ~\Omega_1(\bs) = \big\{ \mathbf{s} ~ | ~ (\ref{det_cap}) \big\} \\
& ~\quad  \quad \Omega_2(\bx,\by,\bq,\bz) = \big\{ (\bx,\by,\bq,\bz) ~ | ~  (\ref{det_supply}) - (\ref{det_decision}) \big\}.\\
& ~\quad \quad  \lambda_{i}^{t} \in \mathcal{D}_2\left(\lambda_{i}^{[t-L:t-1]}\right), ~ \forall i,t \label{aro_uncertainty} .
\end{align}
\end{subequations}
where $\Omega_1$ and $\Omega_2$ are the feasible sets of the first-stage variables $(\bs)$ and second-stage variables $(\bx,\by,\bq,\bz)$, respectively.

\section{Solution Approach}
\label{solution}

This section introduces an iterative algorithm for solving our proposed ARO model, characterized by a complex tri-level ``min-max-min" structure and intricate dependencies between first and second-stage variables. It is crucial to emphasize that while incorporating the binary variables $\bq$ (for service download decision) and $\bz$ (for service placement decision) into our robust framework enhances the characterization of actual operational conditions, it introduces a new layer of complexity in solving the problem. Specifically, these binary recourse variables lead to a non-convex formulation of the second-stage problem, rendering applying the strong duality theorem and the Karush-Kuhn-Tucker (KKT) conditions unfeasible.

To address this challenge, 
we leverage the unique structure of our two-stage RO model, allowing us to approach the inner recourse problem as a two-stage ``max-min" problem. 
We introduce a novel iterative 
algorithm (\textit{ROD}) 
designed to effectively decompose the overall problem into two levels: an outer-level problem (discussed in Section \ref{Master}) and an inner-level problem (outlined in Section \ref{subproblem}). Both levels are suitable for applying the variable and constraint generation method to refine the solution space. Importantly, the resulting subproblem from this decomposition is convex and can be reformulated into a single-level optimization problem using either the strong duality theorem or the KKT conditions.


\vspace{-0.3cm}
\subsection{Robust Optimal Dynamic Algorithm (\textit{ROD})}
We can observe that when the variables $\textbf{q}$ and $\textbf{z}$ are fixed, the innermost minimization problem in (\ref{aro_model}) simplifies to a linear, and thus convex, optimization problem.  This allows us to apply techniques such as CCG \cite{CCGARO} to solve the resulting problem. Hence, a straightforward approach might involve solving (\ref{aro_model}) for every possible value of $\textbf{q}$ and $\textbf{z}$ and selecting the solution with the lowest optimal value. However, this approach is impractical due to the exponential number of possible values of  $\textbf{q}$ and $\textbf{z}$. Therefore, we propose a smart enumeration approach that incrementally adds the most significant value of the integer variables in each iteration until convergence is achieved. Our experiments demonstrate that the proposed algorithm typically converges within a small number of iterations.


The following elaborates on the underlying ideas, with the detailed algorithm presented in Section \ref{alg_detail}. 
First, note that (\ref{aro_model}) can be decomposed into an outer-loop problem and an inner-loop problem. The outer-loop problem remains within the two-stage RO framework, which can be further decomposed into the outer master problem and the inner ``max-min" subproblem. Let $\mathcal{D}^{*} = \{\lambda^1,\lambda^2,\dots,\lambda^l,\dots,\lambda^{K}\}$ be the set of $K$ extreme points of $\mathcal{D}_2$, where $\lambda^l$ is the $l$-th extreme points of set $\mathcal{D}_2$, and $\lambda^l = (\lambda_1^{l},\lambda_2^{l},\dots,\lambda_I^{l})$. By enumerating all elements of $\mathcal{D}^{*}$, the original problem $(\mathcal{P}_1)$ is equivalent to: 
\vspace{-0.2cm}
\begin{subequations}
\label{reformulation1}
\begin{align}
& \underset{s,q,z,x,y}{\text{min}} ~ \mathcal{C}^1 + \mathbf{\eta} \\
& \text{s.t.} ~~~~ \eta \geq \sum_{j,t} \sum_{m \in \mathcal{J} \setminus \{j\}} h_{m,j} q_{m,j}^{t,k} + \sum_{j,t} h_{0,j} q_{0,j}^{t,k} +  \sum_{j,t} \varsigma_j^{t} z^{t,k}_{j}  \nonumber \\
&  + \sum_{j,t} \delta \big[ e_{j}^{t}  y_{j,t}^{B,k} - a_{j}^{t} y_{j,t}^{S,k} \big] + \sum_{t} \delta \big[ e_{0}^{t} y_{0,t}^{B,k} - a_{0}^{t} y_{0,t}^{S,k} \big]  \nonumber \\
& + \sum_{i,j,t} c_{i,j} x_{i,j}^{t,k} + \sum_{i,t} c_{i,0} x_{i,0}^{t,k},  ~ \forall k \leq K \\
& s_{j}^{t}, \: s_{0}^{t} \in \Omega_{1} (\bs)\\
& \lambda_{i}^{t,k} \in \mathcal{D}_2\big(\lambda_{i}^{[t-L:t-1]}\big), ~ \forall k \leq K \\
&  (\bq^{k},\bz^k, \bx^{k}, \by^{k}) \in \Omega_2(\bq,\bz,\bx,\by,\bs,\blambda), ~ \forall k \leq K.
\end{align}
\end{subequations}

When the size of $\mathcal{D}^{*}$ is large, solving the MILP in (\ref{reformulation1}) becomes impractical. Additionally, identifying the set of 
$K$ extreme points of the uncertainty set $\mathcal{D}_2$ is a complex task, making it challenging to search for the optimal solution among all extreme points. Therefore, we develop an iterative approach to solve $(\mathcal{P}_1)$, or equivalently, problem (\ref{reformulation1}). Specifically, at each iteration, we consider a relaxed version of (\ref{reformulation1}) with only a subset of elements of $\mathcal{D}^{*}$. 
The optimal value obtained from this relaxed outer master problem serves as a lower bound (LB) for the original problem $(\mathcal{P}_1)$. Additionally, the optimal solution to this relaxed problem is used as input to the inner \textit{max-mi}n problem in (\ref{aro_model}), which subsequently generates a new extreme point of the uncertainty set $\mathcal{D}_2$ and provides a valid upper bound (UB) for problem $(\mathcal{P}_1)$. 
This newly generated extreme point is then used to add a new constraint to the relaxed master problem. 
By gradually adding more constraints 
to the outer master problem (MP), the LB value becomes non-decreasing. Concurrently, by definition, the UB value is non-increasing. Thus, intuitively, the UB and LB values will converge to the optimal value of $(\mathcal{P}_1)$ after a finite number of iterations.

The inner bilevel \textit{max-min} problem in (\ref{aro_model}) is difficult to solve due to the presence of the integer variables \textbf{q} and \textbf{z}. Thus, we propose converting this \textit{max-min} problem into a trilevel problem and solving it iteratively in a master-subproblem framework, similar to the approach used for the original problem. The following presents the outer master problem, the solution to the inner \textit{max-min} problem, and the \textit{ROD} algorithm.


\vspace{-0.3cm}
\subsection{Outer Master Problem (Outer-MP)}
\label{Master}
The outer level of \textit{ROD} follows an iterative, decomposition-based procedure, involving an iterative interaction between the master problem (MP) and the subproblem.
The initial outer MP (\textbf{Outer-MP}) contains no extreme points. In each iteration, a new extreme point is added to \textbf{Outer-MP}, and the first-stage solution $(\bs)$ is updated. At iteration $k+1$, \textbf{Outer-MP} is: 
\begin{subequations}
\label{MP1}
\begin{align}
& \textbf{Outer-MP}: \min_{\bs,\bq,\bz,\bx,\by} \: \sum_{j,t} \delta  p_{j}^{t} s_{j}^{t} + \sum_{t}  \delta p_{0}^{t} s_{0}^{t} + \mathbf{\eta} \\
& \text{s.t.} ~~ \eta \geq ~ \sum_{j,t} \varsigma_j^{t} z^{t,k}_{j} + \sum_{j,t} \sum_{m \in \mathcal{J} \setminus \{j\}} h_{m,j} q_{m,j}^{t,l} + \sum_{j,t} h_{0,j} q_{0,j}^{t,l}  \nonumber \\
& + \sum_{i,t} c_{i,0} x_{i,0}^{t,l} + \sum_{i,j,t} c_{i,j} x_{i,j}^{t,l} + \sum_{j,t} \delta \Big[ e_{j}^{t}  y_{j,t,l}^{B} - a_{j}^{t} y_{j,t}^{S,l} \Big] \nonumber \\
& + \sum_{t} \delta \Big[ e_{0}^{t} y_{0,t}^{B,l} - a_{0}^{t} y_{0,t}^{S,l} \Big] , ~ \forall l \leq  k\\
& s_j^{t,l}, s_0^{t,l} \in \Omega_1(\bs), \forall l \leq k\\
& (\bq^{l},\bz^{l},\bx^{l},\by^{l}) \in \Omega_2(\bq,\bz,\bx,\by,\bs,\mathbf{\blambda}^{l,*}), \forall l \leq k
\end{align}
\end{subequations}
Here, $\blambda \!=\! \{ \lambda_{i}^{t,1,*},\lambda_{i}^{t,2,*},\dots,\lambda_{i}^{t,k,*}\}$ is the set of optimal solutions to the subproblem in all previous iterations $l=1,\cdots,k$. The optimal solution to \textbf{Outer-MP} includes the optimal first-stage variables $(\bs^{k+1,*})$, the second-stage cost $(\eta^{l,*})$, workload allocation variables $(\bx^{l,*})$, service download variables $(\bq^{l,*})$, service placement variables $(\bz^{l,*})$ and resource adjustment variables $(\by^{l,*})$, $\forall l \leq k$. The resource reservation decision $(\bs)$ will serve as input to the outer subproblem problem described in Section~\ref{subproblem} below. From \textbf{Outer-MP}, we can derive the LB for the optimal value of the original problem. The LB at iteration $k+1$ is:\vspace{-0.2cm}
\begin{align}
\label{outer_LB}
    LB^{\text{outer}} = \sum_{j,t}& \delta \bigg[ p_{j}^{t} s_{j}^{t,k+1,*} + p_0^{t} s_0^{t,k+1,*}\bigg] + \eta^{k+1,*}.
\end{align}

\vspace{-0.4cm}
\subsection{Outer Subproblem}
\label{subproblem}
The objective of this outer bilevel \textit{"max-min"} subproblem is to determine the optimal solution to the innermost minimization problem under the worst-case scenario within the uncertainty set $\mathcal{D}_2$. However, we cannot directly reformulate the problem using standard techniques such as the strong duality theorem and KKT conditions due to integer recourse variables in the subproblem. To address this, we propose a novel iterative approach for smart enumeration of possible values of the integer variables. In particular, we first transform this bilevel ``max-min" subproblem into an equivalent trilevel model, as shown in (\ref{Tri_equivalent}), by separating the binary decision variables $(\bq, \bz)$ from the continuous variables $(\bx, \by)$. Let $\Xi \!=\! \{\bq, \bz\}_{r = 1}^{\mathcal{R}}$, where $\mathcal{R}$ represents the total number of discrete points in the set $\Xi$ and $r$ represents the index of each point in $\Xi$. Given the input $\hat{\bs}$ from the outer-level problem, we can express 
the bilevel problem as 
\begin{subequations}
\begin{align}
\label{Tri_equivalent}
    & Q(\hat{s}) = \underset{\lambda \in \mathcal{D}_2}{\text{max}}  \underset{q,z}{\text{min}} \sum_{j,t} \!\sum_{m \in \mathcal{J} \setminus \{j\}} \!\!\! h_{m,j}^{t}  q_{m,j}^{t} \!+\!\! \sum_{j,t} h_{0,j}^{t} q_{0,j}^{t} \!+\!\! \sum_{j,t} \varsigma_j^{t} z^{t}_{j} \nonumber\\ 
    &  + \underset{x,y}{\text{min}} \: \sum_{j,t} \delta \Big[ e_{j}^{t}  y_{j,t}^{B} - a_{j}^{t} y_{j,t}^{S} \Big]  + \sum_{t} \delta \Big[ e_{0}^{t} y_{0,t}^{B} - a_{0}^{t} y_{0,t}^{S} \Big]\nonumber\\
    & + \sum_{i,t} c_{i,0} x_{i,0}^{t} + \sum_{i,j,t} c_{i,j} x_{i,j}^{t}  \\
    & \text{s.t.}~~  (\bq,\bz,\bx,\by) \in \Omega_2(q,z,x,t,\lambda,\hat{s})
\end{align}
\end{subequations}

By enumerating all discrete points in $\Xi$, we obtain the equivalent trilevel form as follows:
\begin{subequations}
\begin{align}
    & Q(\hat{s}) = \underset{\lambda \in \mathcal{D}_2}{\text{max}} ~ \theta \\
    & \text{s.t.} ~~ \theta \leq \sum_{j,t} \sum_{m \in \mathcal{J} \setminus \{j\}} h_{m,j}^{t}  q_{m,j}^{t,r} + \sum_{j,t} h_{0,j}^{t} q_{0,j}^{t,r} + \sum_{j,t} \varsigma_j^{t} z^{t,r}_{j} \nonumber\\ 
    & \quad + \sum_{j,t} \delta \Big[ e_{j}^{t}  y_{j,t}^{B,r} - a_{j}^{t} y_{j,t}^{S,r} \Big]  + \sum_{t} \delta \Big[ e_{0}^{t} y_{0,t}^{B,r} - a_{0}^{t} y_{0,t}^{S,r} \Big] \nonumber \\
    & \quad + \sum_{i,t} c_{i,0} x_{i,0}^{t,r} + \sum_{i,j,t} c_{i,j} x_{i,j}^{t,r}, ~ \forall r \leq R \label{eq:thetaUBConstraint}\\
    & (\bq,\bz,\bx,\by) \in \Omega_2(q,z,x,t,\lambda,\hat{s}), \: \forall r \leq R 
\end{align}
\end{subequations}
This formulation exhibits a structure similar to the outer loop, prompting us to explore an iterative, decomposition-based method to solve the inner-loop problem.
Specifically, instead of solving the full equivalent problem for all discrete points in the set $\Xi$, our focus narrows to a subset $\Xi^e \subseteq \Xi$. The resulting relaxed problem produces a upper bound for $Q(\hat{s})$.


Specifically, we decompose the trilevel problem (\ref{Tri_equivalent}) into an inner MP (\textbf{Inner-MP}) and an inner subproblem (\textbf{Inner-SP}). The inner MP deals with the 
optimization problem with fixed binary decision variables $(\hat{z}, \hat{q})$, providing an upper bound $\text{UB}^{\text{Inner}}$ for $Q(\hat{s})$. 
The solution of the inner MP identifies the extreme points that help update the inner subproblem. The inner subproblem calculates optimal decisions $(\bz^*, \bq^*, \bx^*, \by^*)$ and improves the lower bound $\text{LB}^{\text{Inner}}$ for $Q(\hat{s})$. 
Through iterative resolution of the updated \textbf{Inner-MP} and modified \textbf{Inner-SP}, $\text{UB}^{\text{Inner}}$ and $\text{LB}^{\text{Inner}}$ are refined in each iteration. These decisions are then relayed to the outer loop, assisting in the update of the first-stage problem and decisions. We utilize the strong duality theorem to reformulate the innermost problem, as it results in fewer binary variables than using KKT conditions, which was justified in our previous work \cite{resilientEC}. Subsequent sections will delve into the reformulation of the inner problem and outline the process of the iterative \textit{ROD} algorithm.

\subsubsection{Inner Subproblem}
Given uncertain realization $\lambda^{*}$ from \textbf{Inner-MP}, the inner subproblem is expressed as follows: 
\begin{subequations}
\label{Inner_SP}
\begin{align}
& \textbf{Inner-SP}: \: \underset{q,z,x,y,\eta}{\text{min}}  ~ \mathbf{\eta} \nonumber\\
&\text{s.t.} ~~ \eta \geq \sum_{j,t} \sum_{m \in \mathcal{J} \setminus \{j\}} h_{m,j} q_{m,j}^{t} + \sum_{j,t} h_{0,j} q_{0,j}^{t} + \sum_{j,t} \varsigma_j^{t} z^{t}_{j} \nonumber\\ 
& + \sum_{j,t} \delta \Big[ e_{j}^{t}  y_{j,t}^{B} - a_{j}^{t} y_{j,t}^{S} \Big] + \sum_{t} \delta \Big[ e_{0}^{t} y_{0,t}^{B} - a_{0}^{t} y_{0,t}^{S} \Big]\nonumber \\ 
& + \sum_{i,t} c_{i,0} x_{i,0}^{t} + \sum_{i,j,t} c_{i,j} x_{i,j}^{t}  \\
& \Lambda_1(\bq,\bz) = \Big\{ (\ref{constr-det_place}), ~ (\ref{det_var2}) \Big\} \\ 
& \Lambda_2(\bx,\by, \mathbf{\hat{s}}) = \Big\{ (\ref{constr-det_adjust1}) - (\ref{constr-det_adjust2}), (\ref{det_var1}); \\
& x_{i,0}^{t} + \sum_i x_{i,j}^{t} \geq \lambda_{i}^{t,*}, \: \forall i,t \\
&  w \sum_i x_{i,j}^{t} \leq  y^{B}_{j,t} - y^{S}_{j,t} + \hat{s}_{j}^{t}\leq C_j z_{j}^{t}, \: \forall j,t \\
& \hat{s}_{0}^{t} + y^{B}_{0,t} - y^{S}_{0,t} \geq w \sum_i x_{i,0}^{t},  \: \forall t \Big\}.
\end{align}
\end{subequations}
From \textbf{Inner-SP}, we obtain the optimal solution $(\mathbf{q}^{r+1,*}$, $\mathbf{z}^{r+1,*}$, $\mathbf{x}^{r+1,*}$, $\mathbf{y}^{r+1,*}$, $\mathbf{\eta}^{r+1,*})$ and update LB for the inner problem:
\begin{align}
\label{inner_LB}
    & LB^{\text{inner}} = \text{max} \: \Bigg\{LB^{\text{inner}}, \sum_{j,t} \sum_{m \in \mathcal{J} \setminus \{j\}} h_{m,j} q_{m,j}^{t,*} + \sum_{j,t} h_{0,j} q_{0,j}^{t,*} \nonumber \\ 
    & + \sum_{j,t} \delta \Big[ e_{j}^{t}  y_{j,t}^{B,*}  - a_{j}^{t} y_{j,t}^{S,*} \Big] + \sum_{t} \delta  \Big[ e_{0}^{t} y_{0,t}^{B,*} - a_{0}^{t} y_{0,t}^{S,*} \Big] \nonumber \\
    & + \sum_{j,t} \varsigma_j^{t} z^{t,*}_{j} +  \sum_{i,t} c_{i,0} x_{i,0}^{t,*} + \sum_{i,j,t} c_{i,j} x_{i,j}^{t,*} \Bigg\}. 
\end{align}
After the inner problem converges, the extreme points and UB for the inner-loop problem can be obtained. We have:
\begin{align}
\label{UB_outer}
\!\!UB^{\text{Outer}} \!\!=\! \min  \bigg\{\!UB^{\text{Outer}}, \sum_{j,t} \!\delta \!\bigg( p_{j}^{t} s_{j}^{t} +p_{0}^{t}s_{0}^{t} \bigg) \!\!+\! Q(\hat{\mathbf{s}}) \! \bigg\}.\!\!
\end{align}

\subsubsection{Inner Master problem} Given the optimal solution $(\hat{\mathbf{s}})$ from \textbf{Outer-MP}, (\ref{Tri_equivalent}) can be transformed into the following ``max-min-max" problem using the strong duality theorem:
\begin{align}
\label{Inner_MP_obj}
    & Q(\hat{\mathbf{s}}) = \underset{\lambda \in \mathcal{D}_{2}}{\text{max}} ~ \underset{q,z}{\text{min}} \sum_{j,t} \sum_{m \in \mathcal{J} \setminus \{j\}} h_{m,j} q_{m,j}^{t} + \sum_{j,t} h_{0,j}^{t} q_{0,j}^{t} \nonumber\\ 
    & + \sum_{j,t} \varsigma_j^{t} z^{t}_{j} + \underset{\sigma,\epsilon,\pi,\mu,\omega}{\text{max}} \sum_{i,t} \sigma_{i}^{t} \lambda_i^{t} - \sum_{j,t} C_j \phi_{j}^{t} \nonumber \\
    & - \sum_{j,t} \hat{s}_{j}^{t} \epsilon_{j}^{t} - \sum_{t} \hat{s}_0^{t} \epsilon_{0}^{t} - \sum_{j,t} \omega_{j}^{t} C_j \hat{z}_j^{t} + \sum_{j,t} \hat{s}_{j}^{t} \omega_{j}^{t} \nonumber \\
    & - \sum_{j,t} \hat{s}_{j}^{t} \pi_{j}^{t} - \sum_{t}\pi_{0}^{t} s_0^{t}.
\end{align}
Given $(q^{*},z^{*})$ from \textbf{Inner-SP}, the second minimization problem can be reduced, transforming the problem $Q(\hat{s})$ into a single-level maximization problem. Notably, the problem is complicated by the bilinear terms that exist between the dual variable $\sigma_{i}^{t}$ and $g_i^t$, especially given that $g_{i}^{t}$ spans the range from $[-1,1]$. Note that $\Gamma_1$ is an integer number and $\mathcal{D}_2$ is a convex and compact set. Thus, the worst-case scenarios can be achieved when $g$ is either $-1$ or $1$ \cite{resilientEC}. The objective of (\ref{Inner_MP_obj}) involves a nonlinear objective function, binary decision vector $g$, and continuous dual variable $\sigma$. Clearly, $g \in [-1,1]$, and $\sigma$ is non-negative and bounded by some sufficiently large number $M$. Using these bounds, McCormick relaxation can be applied for linearizing the bilinear term (i.e., $\delta = \sigma g$) \cite{McCormick76}. Let $\mathcal{M}_{\delta,g,\sigma}$ denote the set involving the McCormick inequalities for linearizing the bilinear term $(\delta_{i}^{t} = \sigma_i^{t} g_{i}^{t})$. We can reformulate this bilinear term by its convex hull\footnote{The convex hull of a set of points is the smallest convex polygon or polyhedron that contains all the points in the set.} of the set
\vspace{-0.1cm}
\[\left\{(\delta,g,\sigma):\delta = g \sigma,(g,\sigma)\in [g^L,g^U]\times[\sigma^L,\sigma^U]\right\}\]
is given by
\vspace{-0.1cm}
\begin{align*}
\mathcal{M}_{(\delta,g,\sigma)} := & \Big\{(\delta,g,\sigma):~~ - \sigma \le \delta \le \sigma,\\
& \delta \ge - \sigma - M (1 - g) , ~ \delta \le - \sigma  + M(1 - g) \Big\}.
\end{align*}
where $M$ is a sufficiently large number. Then, the reformulated inner MP at inner iteration $n$ is expressed as follows:
\vspace{-0.1cm}
\begin{subequations}
\label{Inner-MP}
\begin{align}
    & \textbf{Inner-MP}: \: \max_{\lambda,\delta,\sigma,\epsilon,\pi,\mu,\omega,\tau} \tau \\
    & \tau \leq \sum_{j,t} \sum_{m \in \mathcal{J} \setminus \{j\}} h_{m,j}^{t}  \hat{q}_{m,j}^{t} + \sum_{j,t} h_{0,j}^{t,n} \hat{q}_{0,j}^{t,n} + \sum_{j,t} \varsigma_j^{t} \hat{z}^{t,n}_{j} \nonumber\\ 
  &    + \sum_{i,t} \sigma_{i}^{t,n} \bar{\lambda}_i^{t} + \sum_{i,t} \Tilde{\lambda}_{i}^{t} \delta_{i}^{t,n} - \sum_{j,t} C_j \phi_{j}^{t,n} - \sum_{j,t} \hat{s}_{j}^{t} \epsilon_{j}^{t,n} \nonumber \\
    &  - \sum_{t} \hat{s}_0^{t} \epsilon_{0}^{t,n} - \sum_{j,t} \mu_{j}^{t,n} C_j \hat{z}_j^{t} + \sum_{j,t} \hat{s}_{j}^{t} \mu_{j}^{t,n} \nonumber \\
    & - \sum_{j,t} \hat{s}_{j}^{t} \pi_{j}^{t,n} - \sum_{t}\pi_{0}^{t,n} \hat{s}_0^{t}, ~\forall n \leq r\\
    & \text{s.t.} ~~ \pi_{0}^{t,n} \leq \delta e_{0}^{t}, ~  \forall t,n \leq r \\
    & - \epsilon_{0}^{t,n} - \pi_{0}^{t,n} \leq - \delta a_{0}^{t}, ~ \forall t,n \leq r \\ 
    & - \phi_{j}^{t,n} - \mu_{j}^{t,n} + \pi_{j}^{t,n} \leq \delta e_{j}^{t}, ~ \forall j,t,n \leq r \\
    & - \epsilon_{j}^{t,n} + \mu_{j}^{t,n} - \pi_{j}^{t,n} \leq - \delta a_j^{t}, ~  \forall j,t,n \leq r \\
    & - \pi_{0}^{t,n} w + \sigma_{i}^{t,n} \leq c_{i,0}, ~  \forall i,t,n \leq r\\
    & - \pi_{j}^{t,n} w + \sigma_i^{t,n} \leq c_{i,j}, ~  \forall i,j,t,n \leq r\\
    & \delta_{i}^{t,n} \in \mathcal{M}_{\sigma,g,\delta}, ~ \forall i,t, n\leq r \\
    & (\pi^n, \delta^n,\epsilon^n,\sigma^n,\mu^n,\omega^n) \geq 0, ~  \forall n \leq r \\
    & \sum_{i} g_{i}^{t} \leq \Gamma_1, ~ \forall t .
\end{align}
\end{subequations}
Here, $(\pi^n, \delta^n,\epsilon^n,\sigma^n,\mu^n,\omega^n)$ represent the dual variables associated with the constraints of the innermost problem. \textbf{Inner-MP} is an MILP that can be computed by the existing solvers. By solving \textbf{Inner-MP}, 
we can then obtain the optimal solution for $\lambda_{i}^{t}$, denoted by the vector $\lambda^{r+1,*}$. This optimal solution is then used to update the UB $(UB^{\text{inner}})$ of $Q(\hat{s})$, which implies:
\begin{align}
\label{UB_update}
    UB^{\text{inner}} = \tau^{*}.
\end{align}
Furthermore, the obtained  $\lambda^{r+1,*}$ serves as an input for the next iteration of \textbf{Inner-SP}. As the UB $(UB^{\text{inner}})$ and LB $(LB^{\text{inner}})$ of the inner problem converge, the optimal solution is used to update the bounds $UB^{\text{outer}}$ and $LB^{\text{outer}}$ of the outer problem.

\vspace{-0.2cm}

\subsection{Algorithm Overview}
\label{alg_detail}
 \begin{algorithm}[t!]
 \caption{Outer-loop Algorithm}
 \label{ALG:outer_loop}
 \begin{algorithmic}
 \renewcommand{\algorithmicrequire}{\textbf{Input:}}
 \renewcommand{\algorithmicensure}{\textbf{Output:}}
 \REQUIRE Outer bound gap $\epsilon_1$, $\mathcal{\lambda}^1$
 \ENSURE  Optimal solution $\mathbf{s}^{*} = (s_{j}^{t},s_{0}^{t})$\\
 \textbf{Initialization} : $LB^{\text{outer}} = -\infty$, $UB^{\text{outer}} = +\infty$, $k = 0$\\
 \REPEAT
 
 \STATE \textbf{\textit{Step 1:}}\:Solve the \textbf{Outer-MP} in (\ref{MP1}) and derive the optimal solution $(\mathbf{s}^{k+1,*},\mathbf{\eta}_{k+1}^{*})$, and successively update the LB based on (\ref{outer_LB}).

 \STATE \textbf{\textit{Step 2:}}\: Solve the inner bilevel ``max-min" subproblem and call \textbf{Inner-loop Algorithm} in \ref{inner_CCG} to get the optimal solutions $(\mathbf{\lambda}^{*}$,\:$\mathbf{\tau}^{*})$. Update $UB^{\text{Outer}}$ by (\ref{UB_outer}).

\UNTIL{$\frac{UB^{\text{outer}} - LB^{\text{outer}}}{UB^{\text{outer}}} \leq \epsilon_1$}
\RETURN $\mathbf{s}^{*} = (s_{j}^{t},s_{0}^{t})$
\end{algorithmic}
\end{algorithm}

\subsubsection{Outer-loop Algorithm}
The procedure to solve the outer problem is outlined in \textbf{Algorithm \ref{ALG:outer_loop}}. The outer loop aims to identify first-stage decisions $\mathbf{s}^{*}$. Note that there is no identified $\lambda^{*}$ in the first iteration. Consequently, we adopt the methods introduced in \cite{Duong_iot} to generate an extreme scenario by maximizing the total demand during the initial iteration. Without loss of generality, we define $\mathbf{\lambda}^1$ as the extreme demand scenario within the set $\mathcal{D}_2$, characterized by the highest total demand. This particular $\mathbf{\lambda}^1$ can be determined as the optimal solution to the subsequent optimization problem:
\begin{subequations}
\label{extreme_ray_gen}
\begin{align}
    & \max_{g, \lambda} ~ \sum_{i,t} \lambda_{i}^{t} \\
    & \text{s.t.} ~~  \lambda_{i}^{t} = \bar{\lambda}_{i}^{t} + \Tilde{\lambda}_{i}^{t}, \forall i,t  \\
    & \qquad \Tilde{\lambda}_{i}^{t} = \sum_{s = 1}^{L} A^{s}_{i} \Tilde{\lambda}_{i}^{t - s} +  \mathbf{B}_{i}  g_{i}^{t}, ~ \forall i,t\\
    & \qquad \sum_{i} |g_{i}^{t}| \leq \Gamma_1, ~ \forall t ; ~ g_{i}^{t} \in [-1, 1], \forall i,t .
\end{align}
\end{subequations}

Note that our proposed model $(\mathcal{P}_1)$ is equivalent to the problem (\ref{reformulation1}), incorporating constraints across all extreme points of $\mathcal{D}_2$. This set of extreme points includes $\mathbf{\lambda}^1$, as $\mathbf{\lambda}^1$ satisfies all constraints in the uncertainty set $\mathcal{D}_2$. Considering $\mathbf{\lambda}^1$ during the initial iteration can expedite the search process.

 \begin{algorithm}[t!]
 \caption{Inner-loop Algorithm}
\label{ALG:inner_CCG}
 \begin{algorithmic}
 \renewcommand{\algorithmicrequire}{\textbf{Input:}}
 \renewcommand{\algorithmicensure}{\textbf{Output:}}
 \REQUIRE Inner bound gap $\epsilon_2$, $\bs^{*}$ = $(s_{j}^{t}$,~$s_0^{t})$.
 \ENSURE  Optimal solution $(\mathbf{\lambda}^{*},\mathbf{x}^{*}$,$\mathbf{y}^{*},\mathbf{q}^{*},\mathbf{z}^{*},\mathbf{\tau}^{*})$\\
 \textbf{Initialization} : $LB^{\text{inner}} = -\infty$, $UB^{\text{inner}} = +\infty$, $r = 0$\\
 \REPEAT 
 \STATE \textbf{\textit{Step 1:}}\:Solve the \textbf{Inner-SP} in (\ref{Inner_SP}) and derive the optimal solution $(\mathbf{q}^{r+1,*},\mathbf{x}^{r+1,*},\mathbf{z}^{r+1,*},\mathbf{y}^{r+1,*},\mathbf{\eta}^{r+1,*})$, and successively update the LB $(LB^{\text{inner}})$ according to (\ref{inner_LB}).

\STATE \textbf{\textit{Step 2:}}\: Solve \textbf{Inner-MP} in (\ref{Inner-MP}) and update UB $(UB^{\text{inner}})$ of inner problem according to (\ref{UB_update}). Obtain an optimal solution for $\lambda_{i}^{t}$, denoted it by $\mathbf{\lambda_{r+1}^{*}}$.

\STATE \textbf{\textit{Step 3:}}\: Update $\lambda_{r+1} = \mathbf{\lambda_{r+1}^{*}}$ and update $r = r+1$\\
\UNTIL{$\frac{UB^{\text{inner}} - LB^{\text{inner}}}{UB^{\text{inner}}} \leq \epsilon_2$}.

\RETURN $(\mathbf{\lambda}_{i}^{t,*}$,\:$\mathbf{\tau}^{*})$

\end{algorithmic}
\end{algorithm}

\subsubsection{Inner-loop Algorithm}
\label{inner_CCG} With the first-stage variable $(\hat{\mathbf{s}})$ as input, the inner-loop algorithm focuses on solving the bilevel ``max-min" subproblem, which involves mixed-integer recourse variables. Initially, $UB^{\text{inner}}$ and $LB^{\text{inner}}$ are set to $+\infty$ and $-\infty$, respectively. Utilizing the first-stage variable $(\hat{\mathbf{s}})$, \textbf{Algorithm \ref{ALG:inner_CCG}} is designed to identify the worst-case scenarios $\lambda^{*}$ and objective value of inner problem $\tau^{*}$. These identified uncertain scenarios are then incorporated back into the algorithm. The detailed procedural steps for this process are outlined in \textbf{Algorithm \ref{ALG:inner_CCG}}.

\textbf{\textit{Remark:}} In the flowchart of \textit{ROD} depicted in Fig.~\ref{fig:Flowchart}, the blue block corresponds to the outer level (\textbf{Algorithm~\ref{ALG:outer_loop}}), while the red block corresponds to the inner level (\textbf{Algorithm~\ref{ALG:inner_CCG}}). Given the first-stage decision, the inner problem aims to identify optimal uncertainty realizations based on the current first-stage variables until convergence criteria are met. These updated realizations then inform the outer level, which, in turn, updates the first-stage decision. 

\vspace{-0.2cm}
\begin{figure}[h!]
 	\centering
 	\includegraphics[width=0.47\textwidth,height= 0.3\textheight]{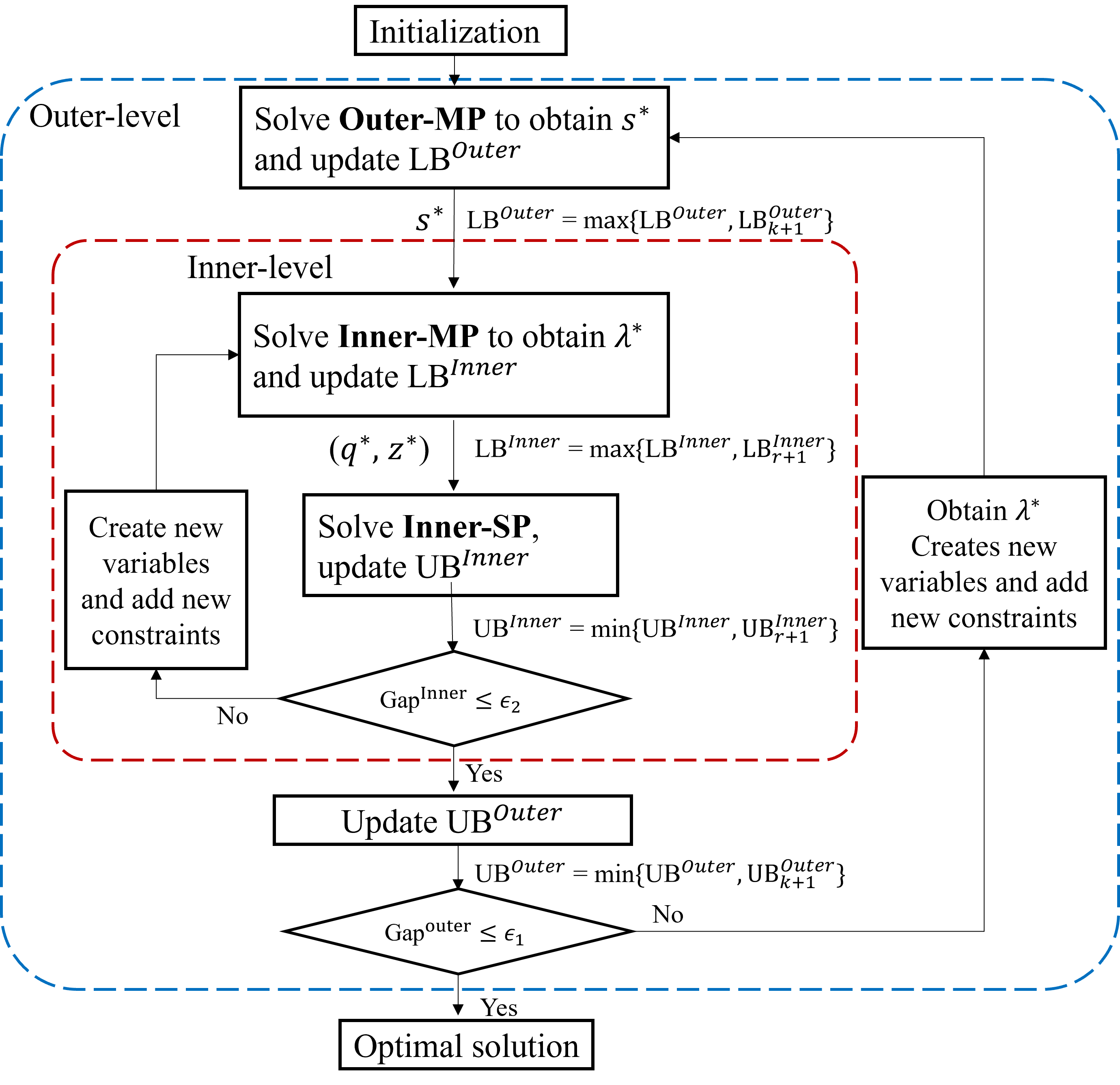} 
 			\caption{Flow diagram of the \textit{ROD} algorithm}
 	\label{fig:Flowchart}
\end{figure} 
 \vspace{-0.2cm}
 
\subsection{Optimality and Convergence Analysis}
The \textit{ROD} algorithm can be viewed as a combination of inner and outer-level iterative methods, depicted in the flowchart in Fig.~\ref{fig:Flowchart}.
Next, we prove that, given first-stage variable $\mathbf{\hat{s}}$, the inner level problem can converge to an immediate optimal within a finite number of steps ($R$) where $\text{UB}^{\text{inner}} = \text{LB}^{\text{inner}}$. The inner-level problem can be regarded as an oracle that provides the UB, subsequently relayed to the outer-level problem.

\begin{proposition}\label{prop:inner}
For any given set of first-stage decisions $\hat{\mathbf{s}}$, the inner-level procedure in \textbf{Algorithm~\ref{ALG:inner_CCG}} exhibits finite convergence. The number of iterations is bounded by $\mathcal{R}$, where $\mathcal{R}$ represents the total number of discrete points in the set $\Xi$.
\end{proposition}
\begin{proof}
This proposition can be shown through contradiction, where the repetition of any $\lambda^*$ implies $UB^{\text{inner}} = LB^{\text{inner}}$. Let $\lambda^{*}$ represent the worst-case uncertainty at iteration $r$ for \textbf{Inner-MP}, and let the optimal decisions obtained by solving \textbf{Inner-SP} be denoted as $(\mathbf{q}^{*}, \mathbf{x}^{*}, \mathbf{z}^{*}, \mathbf{y}^{*}, \eta^{*})$. According to (\ref{UB_update}), we have $UB^{\text{inner}} = Q(\hat{y},\hat{z}) = \theta^{*}$. Thus, \eqref{eq:thetaUBConstraint} implies that
\begin{align}
\label{valid_ine}
    & UB^{\text{inner}} \leq \sum_{j,t} \sum_{m \in \mathcal{J} \setminus \{j\}} h_{m,j} q_{m,j}^{t,*} + \sum_{j,t} h_{0,j} q_{0,j}^{t,*} + \sum_{j,t} \varsigma_j^{t} z^{t,*}_{j} 
    \nonumber\\
    & + \sum_{j,t} \delta \bigg[ e_{j}^{t}  y_{j,t}^{B,*} - a_{j}^{t} y_{j,t}^{S,*} \bigg] + \sum_{t} \bigg[ \delta e_{0}^{t} y_{0,t,r}^{B,*} - a_{0}^{t} y_{0,t}^{S,*} \Bigg]\nonumber \\ 
    & +  \sum_{i,t} c_{i,0} x_{i,0}^{t,*} +  \sum_{i,j,t} c_{i,j} x_{i,j}^{t,*} = RHS
\end{align}
If $\lambda^*$ is repeated, i.e., $\lambda^{*}$ is the worst-case uncertainty at iteration $r+1$ for \textbf{Inner-MP}, the MP at iteration $r+1$ is identical to the MP at iteration $r$. From the update in \eqref{inner_LB} for the lower bound (LB), we have: $LB^{\text{inner}} \geq RHS\geq UB^{\text{inner}}$, implying $LB^{\text{inner}} = UB^{\text{inner}}$. Consequently, Algorithm~\ref{ALG:inner_CCG} converges in a finite number of iterations bounded by $\mathcal{R}$.
\end{proof}


\begin{proposition}
    The outer-level procedure in \textbf{Algorithm~\ref{ALG:outer_loop}} converges an optimal solution of the proposed model $(\mathcal{P}_1)$ in a finite number of iterations. The number of iterations is bounded by $K$, where $K$ is the number of extreme points of $\mathcal{D}^{*}$.
\end{proposition}
\begin{proof}
The proof proceeds in a similar manner to the proof of Proposition~\ref{prop:inner}. According to Proposition \ref{prop:inner}, the inner-level problem converged to its global optimum within maximum $\mathcal{O}(R)$, where $\mathcal{R}$ represents the total number of discrete points in $\Xi$. Note that $\mathcal{D}_{2}$ has a limited number of $K$ elements. To prove this, we only need to show that before convergence, the subproblem always outputs a distinct extreme point at every iteration. In other words, if any extreme point ($\lambda$) is repeated, then $LB = UB$, which means convergence. 

Assume $(\mathbf{s}^{*}, \eta^{*})$ is an optimal solution to the \textbf{Outer-MP} and $(\lambda^{*},\bq^{*},\bx^{*},\bz^{*},\by^{*})$ is an optimal solution to the whole inner problem in iteration $l$. The extreme point $\lambda^{*}$ has appeared in the previous iterations. From the definition of the UB in
\begin{align}
    UB^{\textit{outer}} = \sum_{j,t} \Delta^{T} \big( p_j^{t} s_j^{t,*} + p_0^{t} s_0^{t,*} \big) + Q(\bs^{*})
\end{align}
Since $\lambda^{*}$ appeared in a previous iteration, the cuts related to $\lambda^{*}$ will be added to \textbf{Outer-MP}. Thus, the \textbf{Outer-MP} in iteration $l+1$ is identical to \textbf{Outer-MP} in iteration $l$. Hence, $(\mathbf{s}^{*}, \eta^{*})$ is also the optimal solution to the \textbf{Outer-MP} in iteration $l+1$. Thus, we have: 
\begin{align}
    LB^{\textit{outer}} \geq \sum_{j,t} \delta  p_{j}^{t} s_{j}^{t} + \sum_{t}  \delta p_{0}^{t} s_{0}^{t} + \mathbf{\eta}^{*}
\end{align}
On the other hands, since $(\lambda^{*},\bq^*,\bx^*,\bz^*,\by^*)$ is an optimal solution to the inner problem in iteration $l$. In other words, the valid inequality (\ref{valid_ine}) holds true for every $(\lambda^{*},\bq^*,\bx^*,\bz^*,\by^*)$. From , we have $LB^{outer} \geq UB^{\textit{outer}}$, which implies $LB^{outer} = UB^{\textit{outer}}$. Thus, any repeated extreme points $\lambda^{*}$ implies convergence. Since 
$\mathcal{D}_2$ is a finite set with $K$ elements, the proposed algorithm is guaranteed to converge in $\mathcal{O}(K)$.
\end{proof}
\noindent
Note that \textbf{Algorithm~\ref{ALG:outer_loop}} and \textbf{Algorithm~\ref{ALG:inner_CCG} } typically converge within a few iterations, as demonstrated in the numerical results.

\section{Numerical Results}
\label{Numerical_results}
\subsection{Simulation Setting}
We consider an EC system consisting of $20$ APs and $10$ ENs (i.e., $I=20$, $J=10)$. Similar to \cite{Duong_iot,FairRO_EC}, we generate a random scale-free edge network topology with $100$ nodes using the Barabasi-Albert (BA) model. A subset of nodes is then selected to create an edge network for our analysis. From the resulting network, we can obtain the network hops between nodes $H_{i,j}$ and $H_{i,0}$. The link delay between adjacent nodes in the BA network is randomly generated within the range of $[2,10]$ ms. We employ Dijkstra's shortest path algorithm to compute the network delay $d_{i,j}$ between AP $i$ and EN $j$. 

Based on the pricing 
of \textit{m5d.xlarge} Amazon EC2 instances\footnote{Amazon EC2: https://aws.amazon.com/ec2/pricing/}, the unit resource price $p_j^{t}$ at the ENs is randomly generated within the range $[0.08,0.10] \$$/vCPU.h, while the unit resource price $p_0^{t}$ at the cloud is set to $0.06 \$$/vCPU.h. 
The unit ``\textit{buy-more}" price at EN $j$  is uniformly generated within the interval $[0.08,0.15]$ \$/vCPU.h, while ``\textit{buy-more}" price at the cloud is uniformly generated within the interval $[0.03,0.05]$\$/vCPU.h. Similarly, the unit ``\textit{sell-back}" price at EN $j$,  $a_j^{t}$, takes values randomly from the interval $[0.01, 0.03]$\$/vCPU.h and the unit ``\textit{sell-back}" price at cloud is within the interval $[0.01, 0.02]$\$/vCPU.h. 
The service placement and storage costs $(\varsigma_j^{t})$ are randomly generated between $[0.2,0.3] \$ $. The service placement cost $f_j^{t}$  at EN $j$ at time $t$  is generated from $U[0.1,0.3]$. Additionally, the service download cost $h_{m,j}^{t}$ between ENs is randomly chosen from the interval $[0.05, 0.08]$, whereas the cost between an EN and the cloud is generated within the range $[0.1, 0.3]$. The capacities $(C_j, \forall j)$ of the ENs are set randomly assigned from the set $\{32,48,64\}$ vCPUs. For users' demand, we conduct simulations under two settings. Similar to prior studies \cite{data1,data2,data3}, we leverage data traces provided by Shanghai Telecom \footnote{https://github.com/BuptMecMigration/Edge-Computing-Dataset}, encompassing statistical results from three major cities in China (e.g., Beijing, Shanghai, and Guangzhou). To adapt this data for our EC network topology comprising  $20$ APs and $10$ ENs (i.e., $I=20$, $J=10)$, we segment areas and demand based on geographical districts within those cities and their population information. 

In the \textit{default setting}, we assume that the service is initially unavailable on any EN (i.e., $z_{j}^{\sf 0} = 0$, for all $j$). The system parameters are  $\rho = 0.0001, \Gamma = 5, \beta = 0.02, b = 0.02, w = 0.02$. These parameters will be subject to variation during our sensitivity analyses. All experiments are implemented in MATLAB using CVX and Gurobi on a desktop equipped with an Intel Core i7-11700KF and 32 GB of RAM. Some important cost parameters are summarized in TABLE \ref{table:parameter_setting}.

\begin{table}[t!]
\centering
\begin{tabular}{|l|l|}
\hline
\textbf{Parameters} & \textbf{Values}          \\ \hline
Network size (AP, EN): $(I,J)$ & $(20,10)$          \\ \hline
Service installation cost $f_j^{t}$ & $U[0.1,0.15] \$$ \\ \hline 
Service download cost $h_{m,j}^{t}$ (edge) & $U[0.05,0.08] \$$          \\ \hline
Service download cost $h_{0,j}^{t}$ (cloud) & $U[0.1,0.3]  \$$          \\ \hline
Unit reserved price (edge) $p_{j}^{t}$ & $U[0.08,0.15] $\$$ ~\textit{per hour} $          \\ \hline
Unit reserved price (cloud) $p_{0}^{t}$ & $0.06$\$$ ~\textit{per hour}$           \\ \hline
On-spot price $e_j^{t}$ (edge) & $U[0.1,0.15] $\$$ ~\textit{per hour} $  \\ \hline
On-spot price $e_0$ (cloud) & $U[0.03,0.05] $\$$ ~\textit{per hour} $  \\ \hline
``\textit{Sell-back}" price $a_{j}^{t}$ (edge) & $U[0.01,0.03] $\$$ ~\textit{per hour} $  \\ \hline
``\textit{Sell-back}" price $a_0$ (cloud) & $U[0.01, 0.02] $\$$ ~\textit{per hour} $  \\ \hline
Resource capacity $(C_j, \forall j)$  & $\{32,48,64\}$ vCPU     \\ \hline
Network link delay $d_{i,j}$ & $U[2, 10]$ ms \\ \hline 
Delay penalty $(\rho)$  & $0.0001$ (\$ per ms)\\ \hline
Uncertain budget $(\Gamma)$ & $\Gamma = 5$ \\ \hline 
\end{tabular}
\caption{Simulation data}\label{table:parameter_setting}
\vspace{-0.3cm}
\end{table}

\subsection{Time series analysis to construct $\mathcal{D}_2$}
\label{time_series_process}
To estimate the model parameters of the dynamic uncertainty set, we aim to fit the historical data for each area in the following time series model:
\begin{subequations}
\begin{align}
    & \lambda_i^t = \bar{\lambda}_i^{t} + \Tilde{\lambda}_i^{t}, ~~ \forall i,t \\
    & \Tilde{\lambda}_i^{t} = \sum_{s = 1}^{L} A_i^{s} \Tilde{\lambda}_i^t + \epsilon_{i}^{t}, ~~\forall i,t.
\end{align}
\end{subequations}

After aggregating the data from the initial ten days of November and computing the averages, a distinct seasonal demand pattern becomes evident when analyzing both weekly and daily trends. This pattern consistently exhibits peak demand hours occurring between $10$ AM and $3$ PM, as depicted in Fig. \ref{fig:hour_pattern} and \ref{fig:week_pattern}, mirroring typical human activities throughout the day. Once these seasonal patterns are identified, the forecasted demand $\lambda_{i}^{t}$ can be estimated. In our scenario, daily and semi-daily seasonality can be leveraged. In this context, using a $20$-minute time interval, we would consider:
\begin{align}\label{forecast_model}
    \Bar{\lambda}_i^{t}  &= \varphi_i^{1} + \varphi_i^{2} \cos\left(\frac{2\pi t}{24 \times 3}\right) + \varphi_i^{3} \sin\left(\frac{2 \pi}{24 \times 3}\right) \nonumber\\
    & + \varphi_i^{4} \cos\left(\frac{2 \pi t}{12 \times 3}\right) + \varphi_i^{5} \sin\left(\frac{2 \pi t}{12 \times 3}\right).
\end{align}

The parameters $\varphi_i^{1},\cdots,\varphi_i^5$ can be estimated using linear regression techniques, as outlined in \cite{Time_series}. Notably, the deviation $\Tilde{\lambda}_i^{t}$ follows a multivariate auto-regressive (AR) process of order $L$. This involves determining the optimal time lag $L$ and AR coefficients $A_i^{s}$. In this model, for each area $i$, the error term $\mathbf{\epsilon}_{t}$ is assumed to follow a normal random distribution with mean $0$ and covariance matrix $\Sigma$. Modern time series analysis tools, such as Matlab, Python, or R, facilitate directly estimating parameters $A_i^{s}$. Additionally, the parameter $B_i$ can be derived using the Cholesky decomposition of the covariance matrix $\Sigma$.

\begin{figure}[h!]
	\subfigure[Hourly pattern]{
		  \includegraphics[width=0.246\textwidth,height=0.10\textheight]{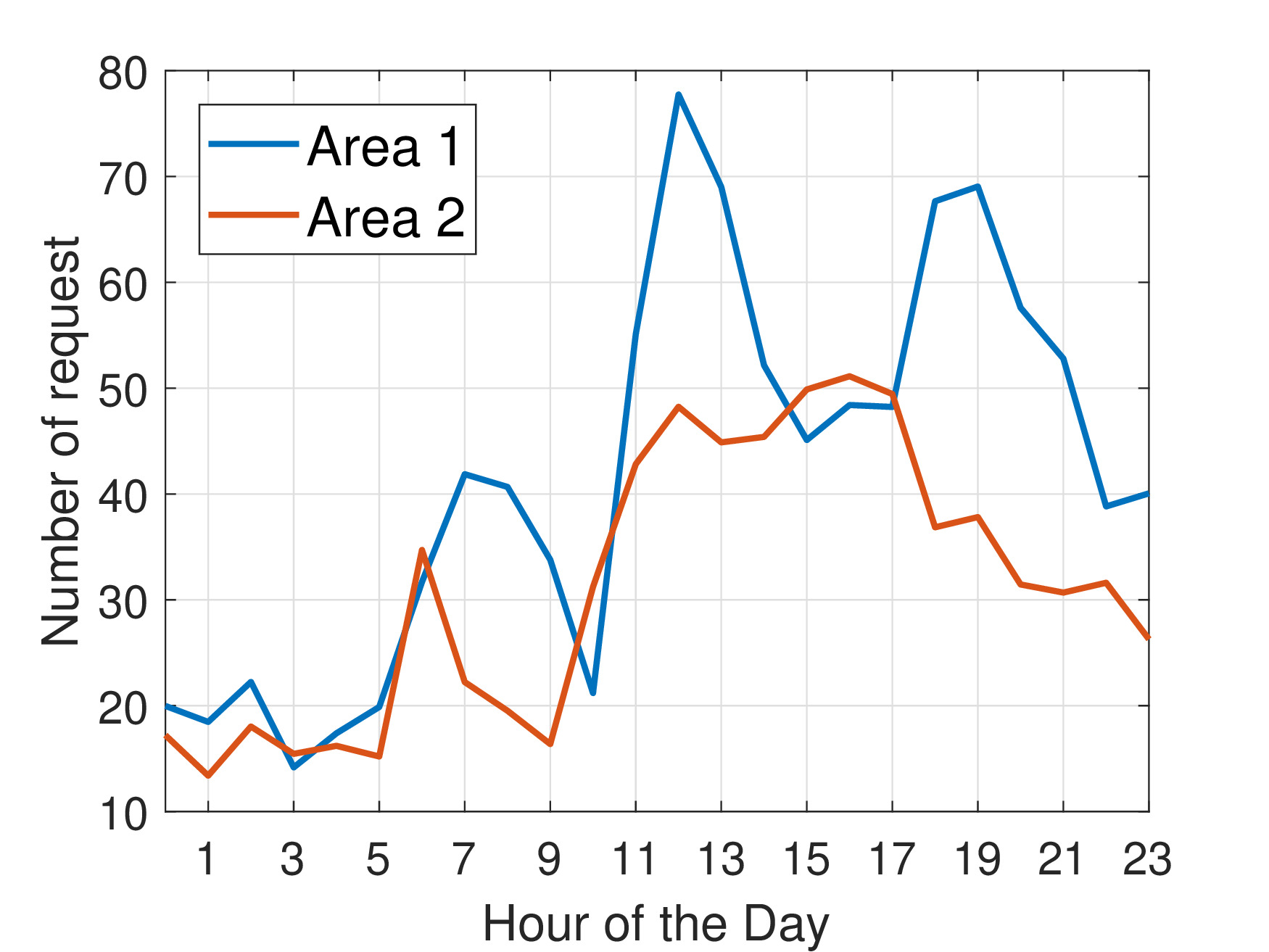}
	    \label{fig:hour_pattern}
	}   \hspace*{-2.1em} 
		 \subfigure[Daily pattern]{
	     \includegraphics[width=0.246\textwidth,height=0.10\textheight]{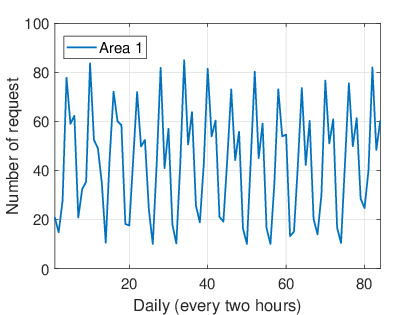}
	     \label{fig:week_pattern}
	}  \vspace{-0.2cm}
	\caption{Time series analysis}
	\vspace{-0.4cm}
\end{figure}

\vspace{-0.3cm}
\subsection{Sensitivity Analysis}
\label{Sensitvity_analysis}
This section 
assesses the impact of key system parameters on the optimal solution. These parameters include the uncertainty factor $\Gamma$, demand deviation $\alpha$, and delay penalty $\rho$. To understand the influence of varying cost parameters, scaling factors $\mathbf{\Psi} = (\Psi_f, \Psi_h, \Psi_p, \Psi_e)$ are introduced. These factors scale up or down the cost parameter $\Psi$ from its default setting, with $\mathbf{\Psi} = 1$ representing the default parameter configuration. For example, by multiplying the base values of $h$ 
by $\mathbf{\Psi}_h$, 
download cost $(h)$ can be increased or decreased accordingly.

\noindent
\textbf{\textit{1) Impacts of the uncertainty set:}} 
Fig. \ref{fig: rho_Gamma} illustrates the impact of the uncertainty set and delay penalty $\rho$ on the optimal solution. Recall that $\Gamma$ governs the conservativeness of the uncertainty set. A scaling factor, $\Psi_{\rho}$, is introduced to adjust the default setting (i.e., $\Psi_{\rho} \in [1,9]$). The figure shows that the total cost increases as 
$\Gamma$ increases since  
the uncertainty set $\mathcal{D}_2$ expands. Noticeably, 
$\rho$ has a considerable influence on the total cost because 
network delay costs are proportional to the allocated workload. With a higher delay penalty $\rho$, the SP adopts a more stringent approach towards service delivery, leading to higher operational costs. 

\vspace{-0.3cm}
\begin{figure}[h!]
\centering         \includegraphics[width=0.30\textwidth,height=0.12\textheight]{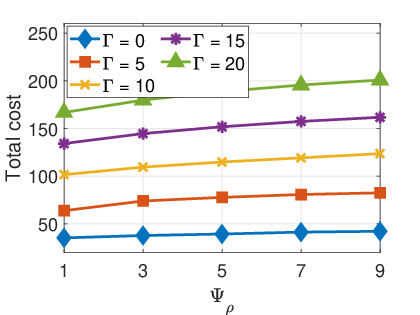}
	   \vspace*{-0.2cm} 
	\caption{Total cost: varying $\rho$ and $\Gamma$}
        \label{fig: rho_Gamma}
	\vspace{-0.3cm}
\end{figure}

\noindent
\textbf{\textit{2) Impacts of cost parameters:}} 
Fig.\ref{fig: pe_cost} - \ref{fig: pe_adjust} demonstrate the impacts of resource reservation cost ($p$) and resource adjustment cost ($e$). Initially, the total cost and payment rise and then reach a plateau after a certain point. This trend occurs as the SP may prefer dynamically procuring resources during the actual operational stage, especially when $p$ is higher. The saturation point indicates that the procured resources are sufficient to meet all user demands. Conversely, with increasing values of $e$, the advantage of dynamic resource adjustment is marginalized. When $e$ is approximately equals 
the resource reservation cost, there is no incentive for resource adjustment, as further illustrated in Fig.\ref{fig: pe_adjust}. Higher values of $\Psi_e$ make the SP more cautious, leading to a delayed point of change in strategy. Additionally, Fig. \ref{fig: fh_cost} reveals that the total cost escalates with increasing service installation cost ($f$) and download cost ($h$). When the cost of downloading services ($h$) is low, the SP is more inclined to actively monitor each EN and strategically select ENs with installed services for placement.

\begin{figure}[h!]
	\subfigure[Total cost: varying $p$ and $e$]{
  \includegraphics[width=0.246\textwidth,height=0.10\textheight]{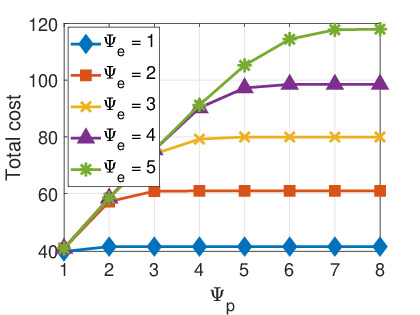}
  	    \label{fig: pe_cost}
	}   \hspace*{-2.1em} 
		 \subfigure[Payment: varying $p$ and $e$]{
    \includegraphics[width=0.246\textwidth,height=0.10\textheight]{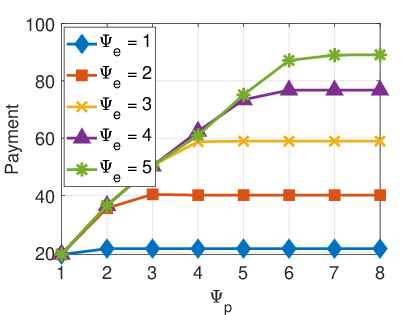}
	 	     \label{fig: pe_payment}
}  \vspace{-0.2cm}
	\subfigure[Adjust: varying $p$ and $e$]{
    \includegraphics[width=0.246\textwidth,height=0.10\textheight]{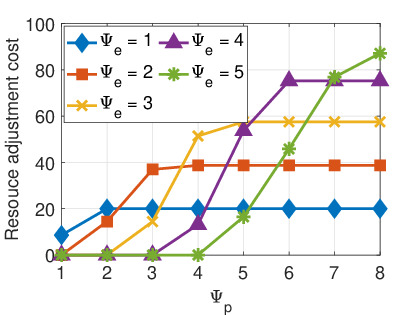}
	 	     \label{fig: pe_adjust}
} \hspace*{-2.1em} 
	\subfigure[Total cost: varying $h$ and $f$]{
    \includegraphics[width=0.246\textwidth,height=0.10\textheight]{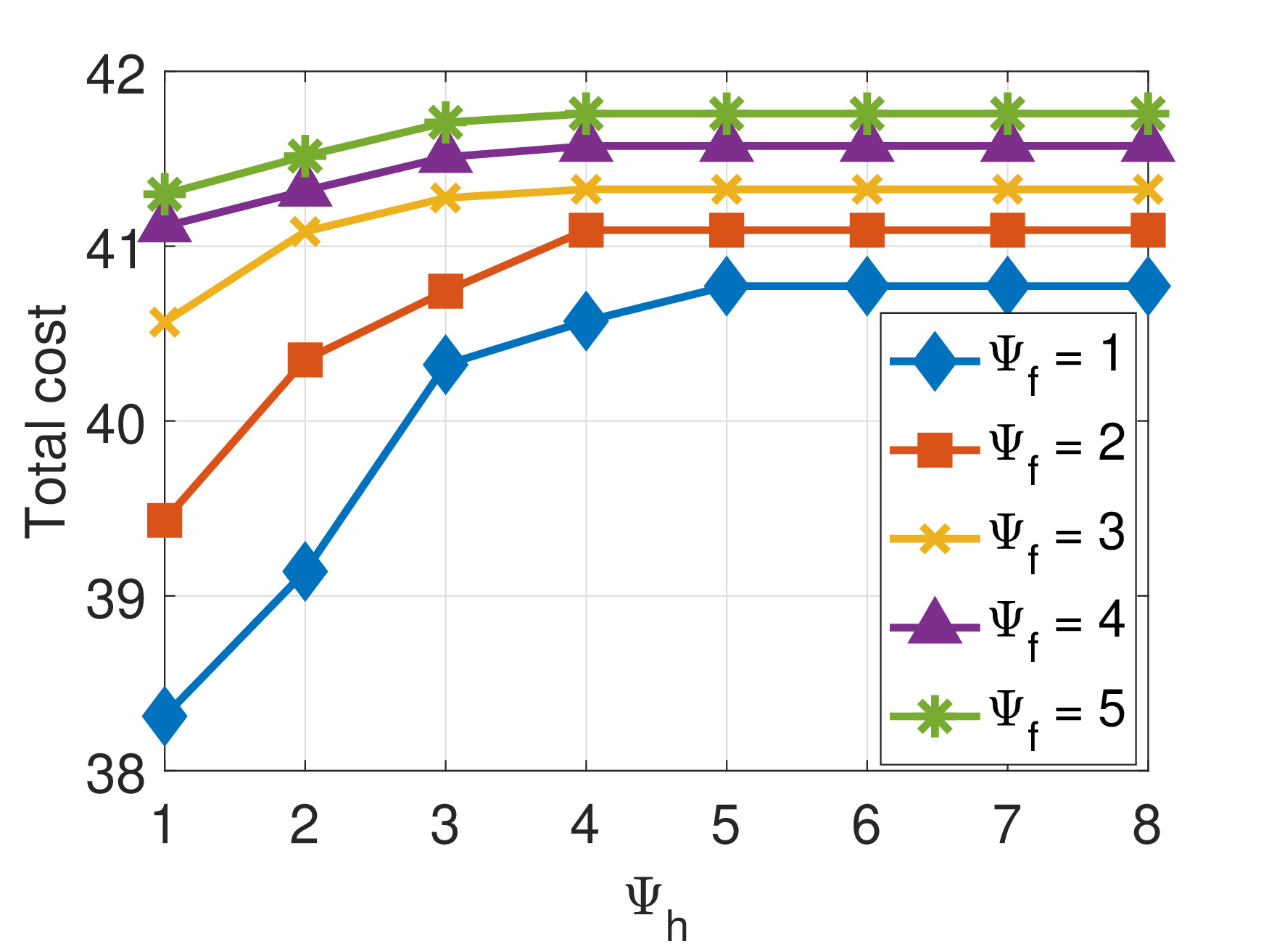}
	 	     \label{fig: fh_cost}
} 
 \caption{The impact of cost parameters}
	\vspace{-0.3cm}
\end{figure}

\noindent
\textbf{\textit{3) Convergence and running time analysis of \textit{ROD}}:}
In our simulations, $\epsilon_1$ and $\epsilon_2$ denote the optimality gaps for the outer and inner loop problems, respectively, both set at $0.1\%$. Firstly, we examine the convergence properties of our proposed algorithm (\textit{ROD}). Figure \ref{fig: outerloop} illustrates the convergence behavior of \textit{ROD} across different values of $\Gamma$. Our observations indicate that the algorithm typically converges within a few iterations in both scenarios. Figure \ref{fig: innerloop} displays the final iterations before the convergence of the outer loop problem.

\begin{figure}[h!]
	\subfigure[$\Gamma = 10$: Outer loop]{
		  \includegraphics[width=0.246\textwidth,height=0.10\textheight]{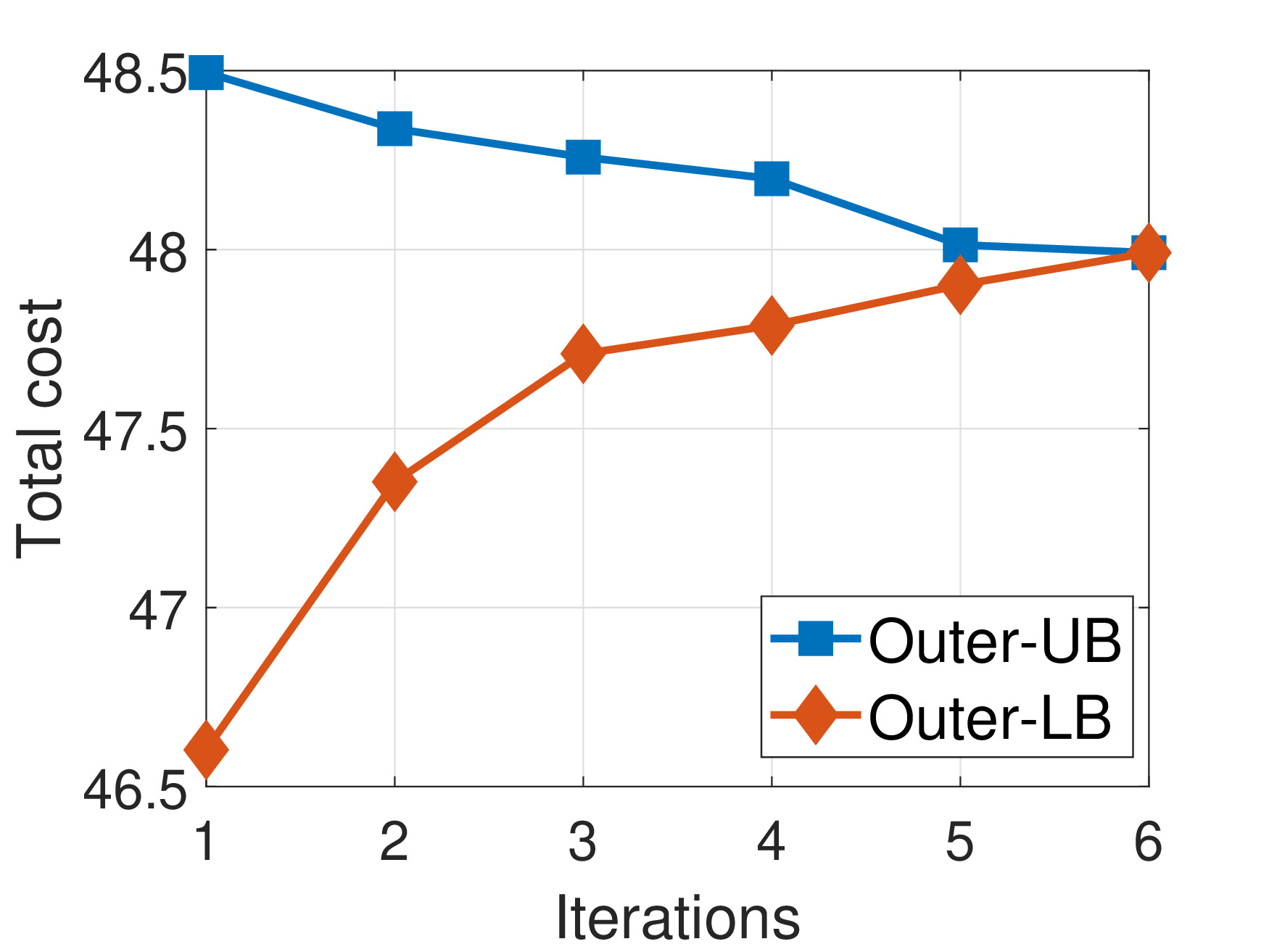}
	    \label{fig: outerloop}
	}   \hspace*{-2.1em} 
		 \subfigure[Inner-loop convergence]{
	     \includegraphics[width=0.246\textwidth,height=0.10\textheight]{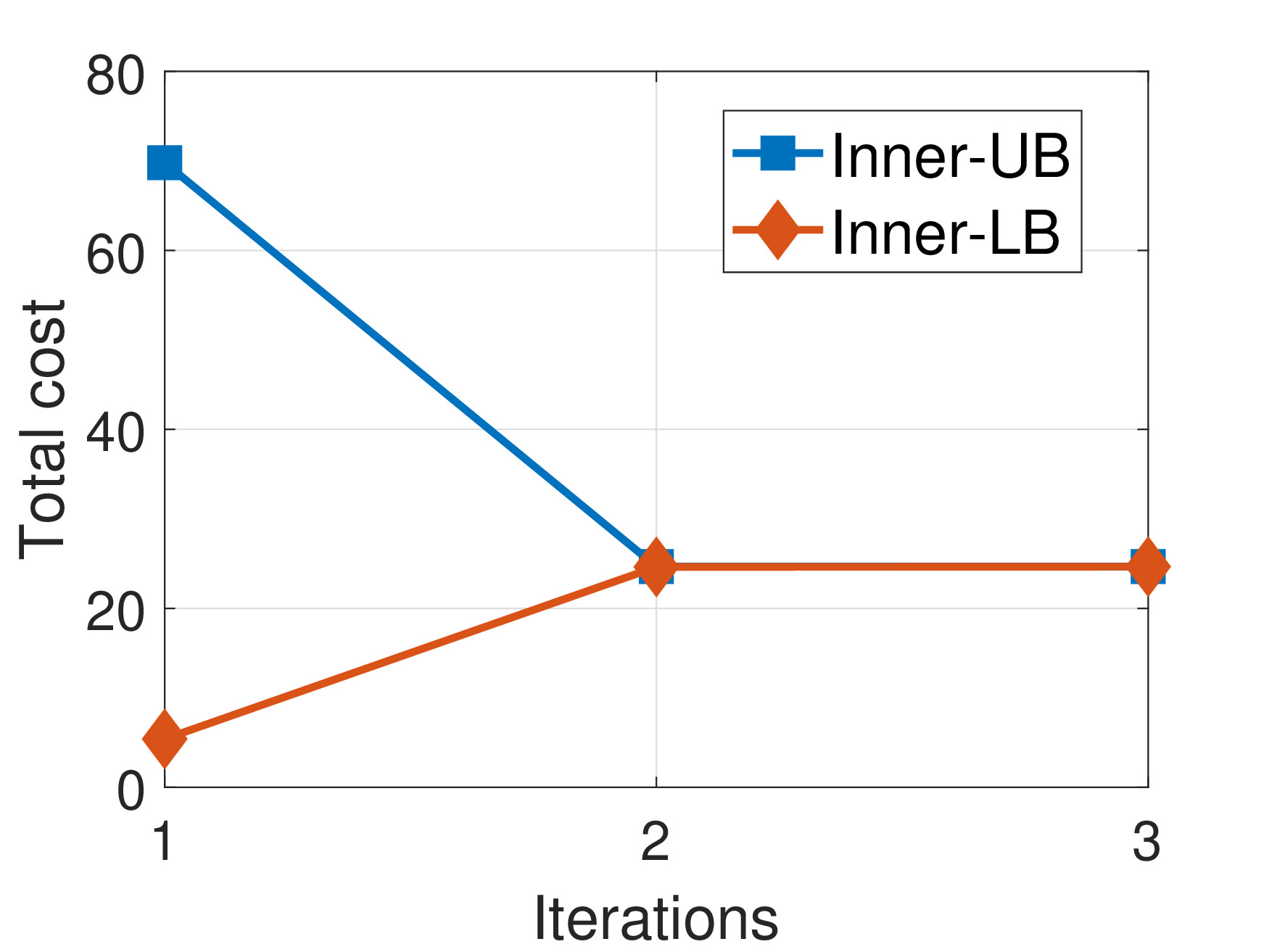}
	     \label{fig: innerloop}
	}  \vspace{-0.2cm}
	\caption{Convergence property}
     \vspace{-0.4cm}
\end{figure} 

For our runtime analysis, we calculate the average computational time across $100$ different problem instances for each problem size. Table \ref{tab: time_table} shows the time taken by our proposed model for various problem sizes and 
uncertain budget $\Gamma$. The increased computational time with larger values of $\Gamma$ is due to the expanded uncertainty set encompassing more extreme points. Therefore, the \textbf{Outer-MP} experiences increased complexity due to the higher dimensionality of the problem, generating more variables in each iteration. Note that the core of the underlying problem is based on planning. The resource reservation decisions are typically made in advance, a day/week/month ahead. However, considering the dynamic nature of demand, which can undergo significant shifts within short intervals, 
the SP may need to make operational decisions more frequently, such as every hour. This approach allows the SP to integrate the latest information and adjust strategies based on real-time conditions. In addition, we can dynamically adjust the optimality between UB and LB to speed up computation, especially for large-scale networks.

\begin{table}[h!]
\centering
\begin{tabular}{|c|c|c|}
\hline
\textbf{Network size}         & $\Gamma$ & D--ARO--DUS \\ \hline
\multirow{2}{*}{$I=20, J=10$} & $5$      &  $263.11$s $(0.1 \%)$         \\ \cline{2-3} 
                              & $10$     & $453.02$s $(0.1 \%)$          \\ \hline
\multirow{2}{*}{$I=20, J=20$} & $5$      &      $447.68$s $(0.1 \%)$  \\ \cline{2-3} 
                                & $10$    & $930.51$s $(0.1 \%)$ \\ \hline
\multirow{3}{*}{$I=20, J=30$} & \multirow{2}{*}{$5$} & $1235.19$s $(0.1 \%)$         \\
                              &                      & $519.19$s $(0.5 \%)$         \\ \cline{2-3} 
                              & \multirow{2}{*}{$10$} & $2110.21$s $(0.1 \%)$        \\
                              &                       &  $946.19$s $(0.5 \%)$            \\ \hline
\end{tabular}
\caption{Running time analysis}
\label{tab: time_table}
\vspace{-0.4cm}
\end{table}

\vspace{-0.2cm}
\subsection{Performance Evaluation}

We will compare the performance of our proposed \textit{D-ARO-DUS} model $(\mathcal{P}_1)$ against the following benchmarks:
\begin{enumerate}
    \item \textit{D-ARO-SUS:} We consider the proposed ARO problem with a static uncertainty set (SUS), where the demand of each AP at each time slot is treated as independent within the SUS. In this context, there is no incorporated correlation sequence, such as (\ref{spatialtemperal}), to capture the spatial-temporal dynamics in the uncertainty set. 
    \item \textit{S-ARO:} (See \textit{Appendix}~\ref{S_ARO_section}) The S-ARO model, as introduced in \cite{Duong_iot}, represents a static ARO framework wherein the SP makes service placement and resource procurement under the worst-case scenarios before the static uncertainty is disclosed. These decisions remain fixed during the operational stage once they are established.
\end{enumerate}
The objective of the SP is to determine optimal planning and operational decisions for the entire horizon. In the static model (\textit{S-ARO}), decisions, such as service placement $(\bz,\bq)$ and resource reservation $(\bs)$ are pre-determined and remain fixed during the operational stage. Recall that our assumption is no available ENs with installed service at the beginning, which means that the SP must download service from the cloud at $t = 1$. For a fair comparison, the service placement cost in \textit{S-ARO} model is calculated by $f_j^{t} + h_0$. This implies that the SP is limited to downloading the service solely from the remote cloud, incurring service download cost $h_0$. In the dynamic models (\textit{D-ARO-SUS} and \textit{D-ARO-DUS}), the SP sets resource reservation $(\bs)$ before uncertainty disclosure, while other decisions, such as service placement and resource adjustment, can be modified during the operational stage. These dynamic models are collectively denoted as $\textit{D-ARO}$. The only distinction between \textit{D-ARO-SUS} and \textit{D-ARO-DUS} lies in whether the uncertainty set captures the spatial-temporal relationship of the demand. Our evaluation will concentrate on three key performance metrics: \textit{total cost}, \textit{payment}, and \textit{cost of resource adjustment}. The \textit{payment} encompasses costs related to resource procurement, adjustment, as well as service placement (download, installation), and storage.  Notably, in the dynamic model, resource adjustment cost is included as part of the overall payment, whereas the static model does not incur costs associated with resource adjustment.

\noindent
\textit{\textbf{1) Impacts of cost parameters:}} 
Figs. \ref{fig: f_cost}--\ref{fig: f_payment} provide a comparative analysis of three schemes by varying cost parameters associated with service installation and download decisions. Generally, the total cost for all schemes increases with higher values of $\Psi_f$. Notably, the \textit{D-ARO} model outperforms the \textit{S-ARO} model. The \textit{S-ARO} model determines service placement decisions based on worst-case scenarios during planning, with no flexibility for adjustments during the operational phase. In contrast, our proposed model allows dynamic selection of the optimal service download location based on users' requests. This advantage is particularly evident when there are no available ENs with installed service at $t = 0$, requiring the SP to download service from the remote cloud. Consequently, the \textit{S-ARO} model may allocate excess resources, even when those ENs may not be immediately needed. Beyond certain $\Psi_f$ values, a saturation effect is observed, where higher $\Psi$ values make $h$ approach $h_0$, indicating that the cost of downloading services from nearby ENs becomes comparable to downloading from the cloud. Consequently, the advantage of dynamic service placement is less pronounced in such scenarios.
\begin{figure}[h!]
	\subfigure[Total cost: varying $h$]{
	  \includegraphics[width=0.245\textwidth,height=0.10\textheight]{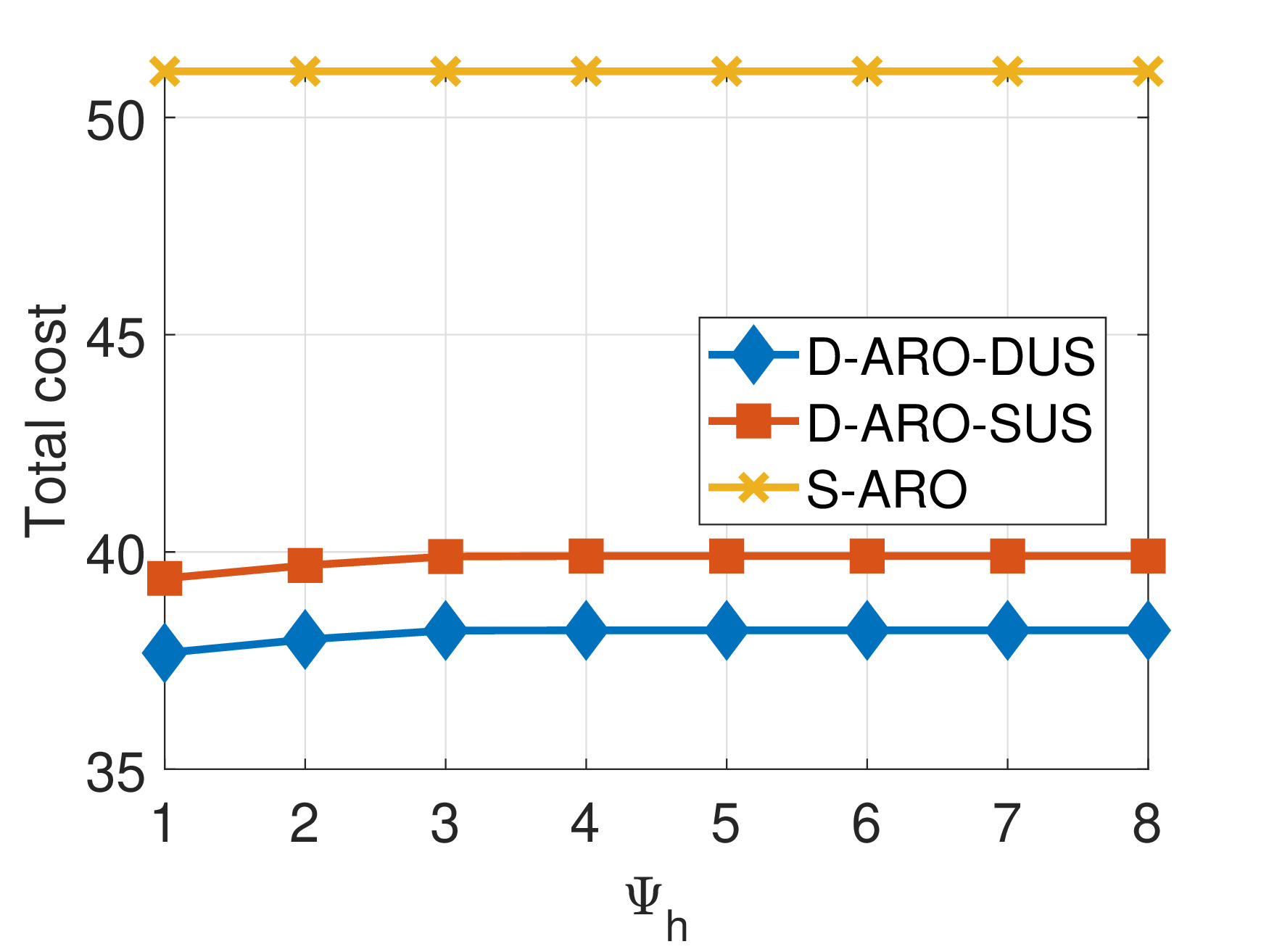}
	    \label{fig: h_cost}
	}   \hspace*{-2.1em} 
		 \subfigure[Payment: varying $h$]{
	     \includegraphics[width=0.245\textwidth,height=0.10\textheight]{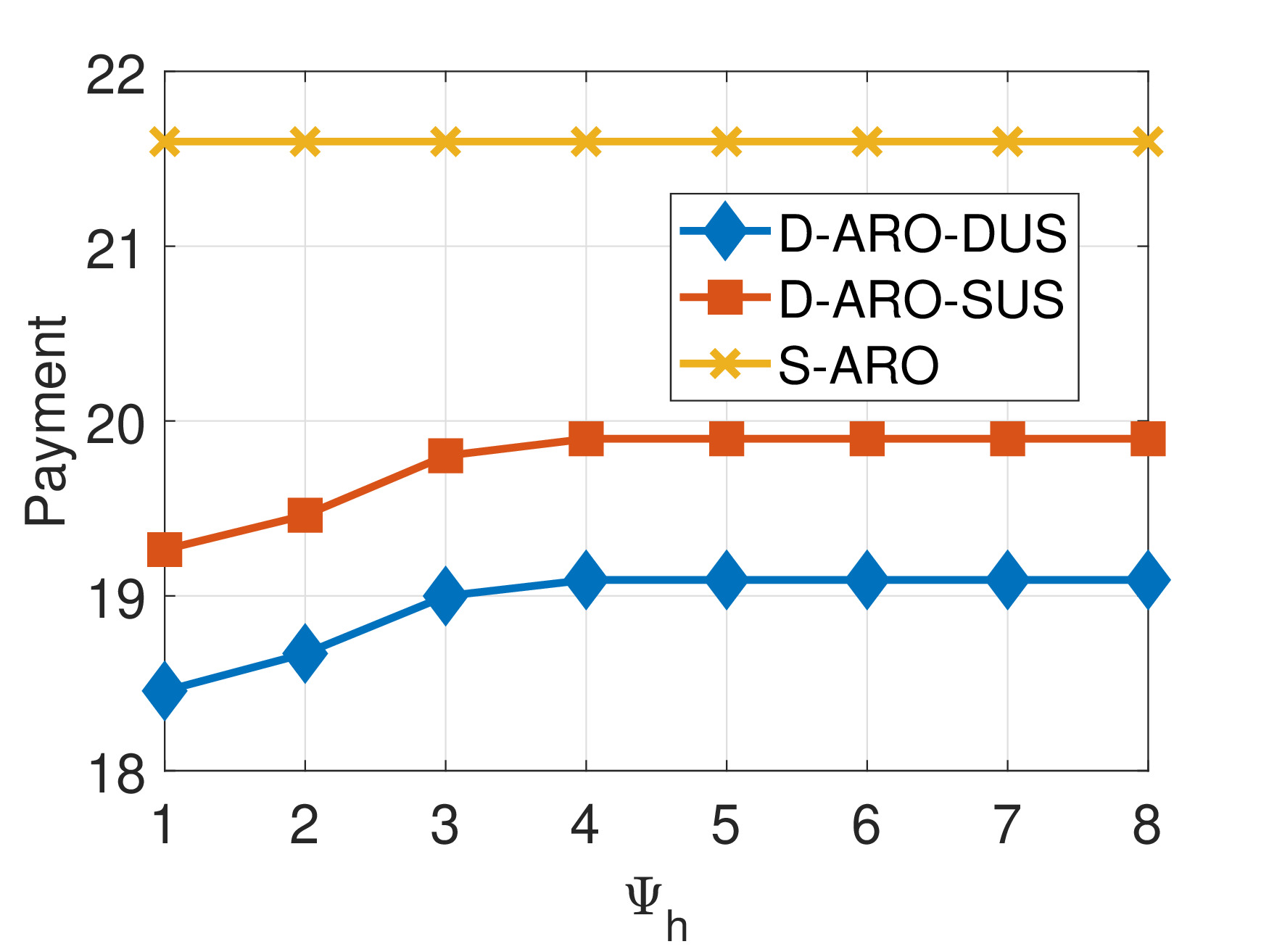}
	     \label{fig: h_payment}
	}  \vspace{-0.3cm}
	\subfigure[Total cost: varying $f$]{
	     \includegraphics[width=0.245\textwidth,height=0.10\textheight]{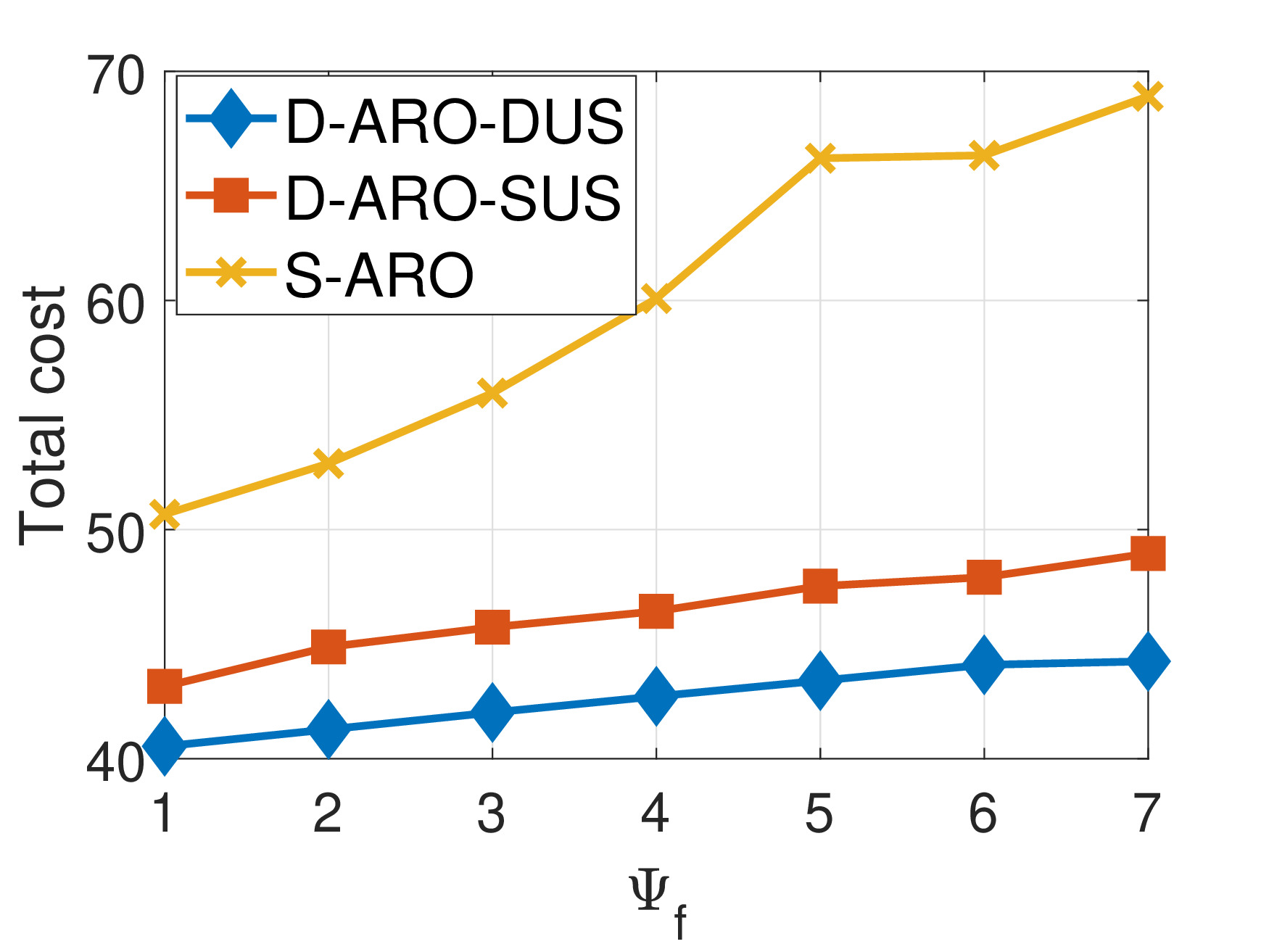}
	     \label{fig: f_cost}
	}  \hspace*{-2.1em}
	\subfigure[Payment: varying $f$]{
	     \includegraphics[width=0.245\textwidth,height=0.10\textheight]{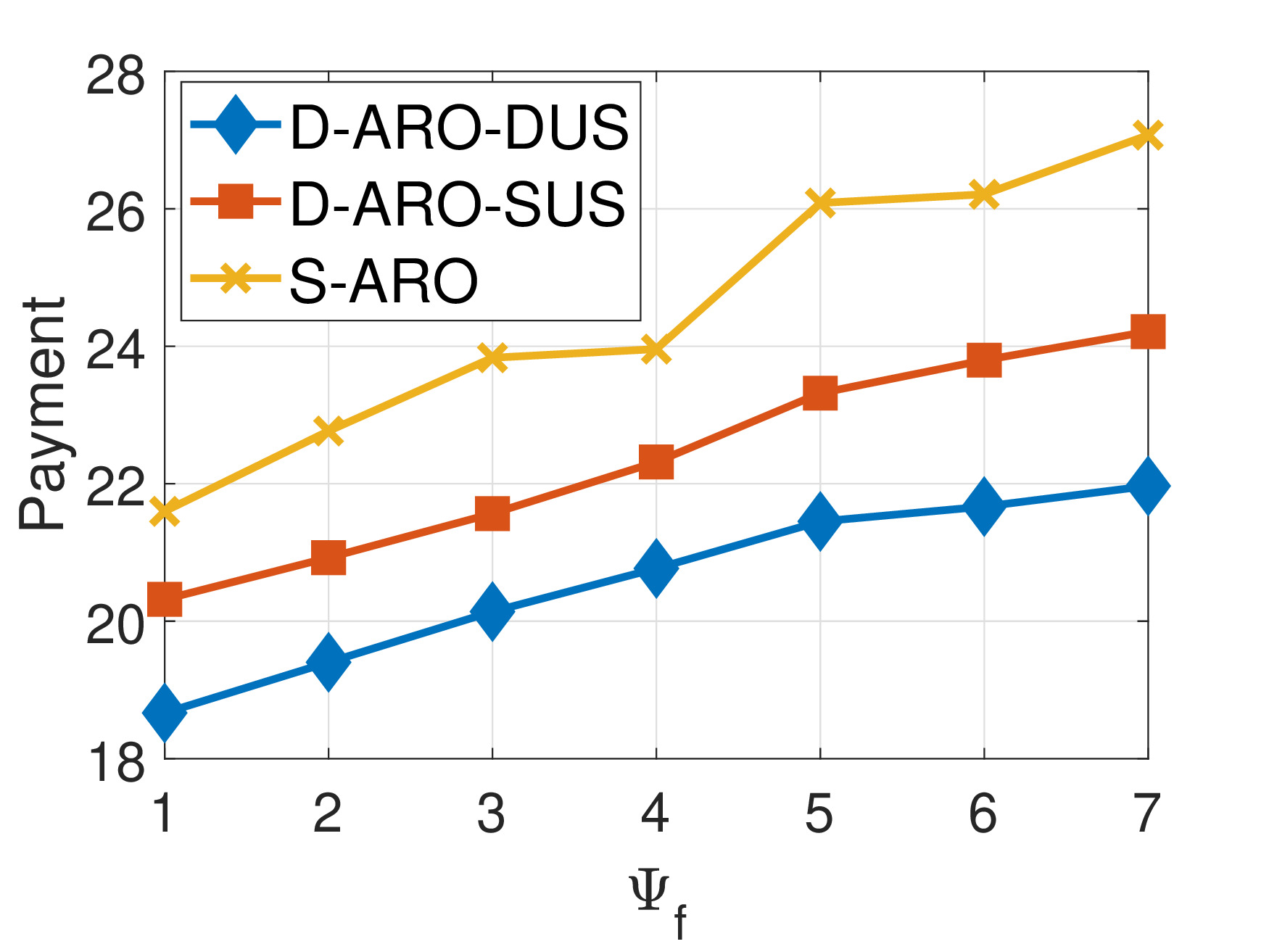}
	     \label{fig: f_payment}
	}  
	\caption{The impact of the cost parameters}
	\vspace{-0.2cm}
\end{figure} 

Unlike the static model, our dynamic model allows the SP to postpone immediate decisions on service download locations. This increased flexibility empowers the SP to exploit favorable conditions, such as lower $h$ costs, and choose to download the service from nearby ENs. When comparing dynamic models, the \textit{D-ARO-DUS} model outperforms the \textit{D-ARO-SUS}, as evident in Figs. \ref{fig: h_payment}--\ref{fig: f_payment}. This improved performance can be attributed to the dynamic uncertainty set's ability to effectively capture spatial-temporal dynamics, leading to less conservative outcomes than those produced using static uncertainty sets.


Figs. \ref{fig: varying_p_model_cost}--\ref{fig: varying_e_adjust} highlight the advantages of resource adjustment in our dynamic model compared to the \textit{S-ARO} model. Considering both long-term resource reservation and short-term resource adjustment, our dynamic model achieves significantly lower resource procurement costs than 
the static model. The interplay among the price parameters $(\mathbf{p}^{t}$, $\mathbf{e}^{t}$, $\mathbf{a}^{t})$ crucially influences the SP's decisions and the resulting resource procurement costs. Specifically, the reserved price $\bp$ is lower than the on-spot price $\be$, and the retail price $\ba$ should be lower than the reserved price $\bp$ (i.e., $\ba \leq \bp \leq \be$). Figs. \ref{fig: varying_p_model_cost}--\ref{fig: varying_p_model_payment} show that when the on-spot price $(\mathbf{e}^{t})$ is significantly higher than the reserved price $(\mathbf{p}^{t})$, the SP leverages the lower reserved price since there is no incentive to purchase additional resources during the operational stage. Conversely, when the on-spot price is close to the reserved price, the SP may not necessarily reserve sufficient resources in advance. This adaptive approach allows the SP to respond to changing market conditions and optimize resource procurement decisions based on the current on-spot price, showcasing the benefits of dynamic resource adjustment. Another interesting phenomenon in Figs. \ref{fig: varying_p_model_cost} and \ref{fig: varying_p_model_payment} is that \textit{S-ARO} model initially increases and then decreases even lower than the dynamic model. The initial increase indicates the SP is still willing to place the service onto ENs. However, after the resource reservation cost becomes high enough, the SP would rather place the service and procure all resources from the remote cloud, leading to a decrease in the flat on payment but a significant increase in the total cost. An intriguing observation from Figs. \ref{fig: varying_p_model_cost}--\ref{fig: varying_p_model_payment} is the initial rise and subsequent fall in costs for the \textit{S-ARO} model, with costs eventually dipping below those of dynamic models. The initial increase suggests that the SP is more inclined to place the service onto ENs. However, as the resource reservation cost escalates, the SP tends to shift towards placing the service and procuring all resources from the remote cloud. This strategic shift results in a decrease and leveling off in payment costs, but it also leads to a notable rise in the total cost. This trend reflects the SP' strategy in response to changing costs, yet notably overlooks the aspect of user experience. This observation underscores the advantages of dynamic models in effectively balancing cost considerations with user experience. 

\vspace{-0.2cm}
\begin{figure}[h!]
	\subfigure[Total cost: varying $p$]{
		\includegraphics[width=0.246\textwidth,height=0.10\textheight]{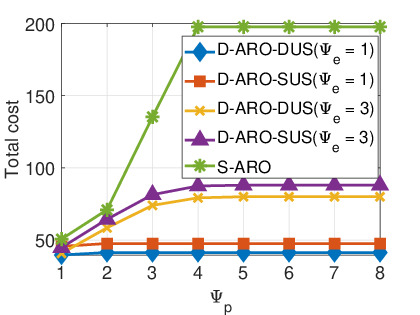}
	    \label{fig: varying_p_model_cost}
	}   \hspace*{-2.1em} 
		 \subfigure[Payment: varying $p$]{
	    \includegraphics[width=0.246\textwidth,height=0.10\textheight]{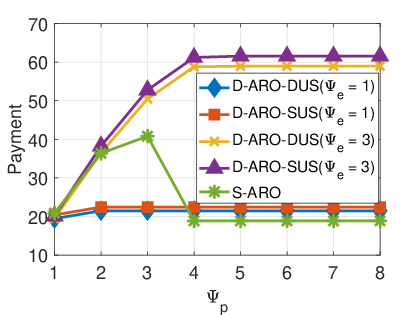}
	     \label{fig: varying_p_model_payment}
	}  \vspace{-0.3cm}
 	\subfigure[Resource adjustment: varying $p$]{
	    \includegraphics[width=0.246\textwidth,height=0.10\textheight]{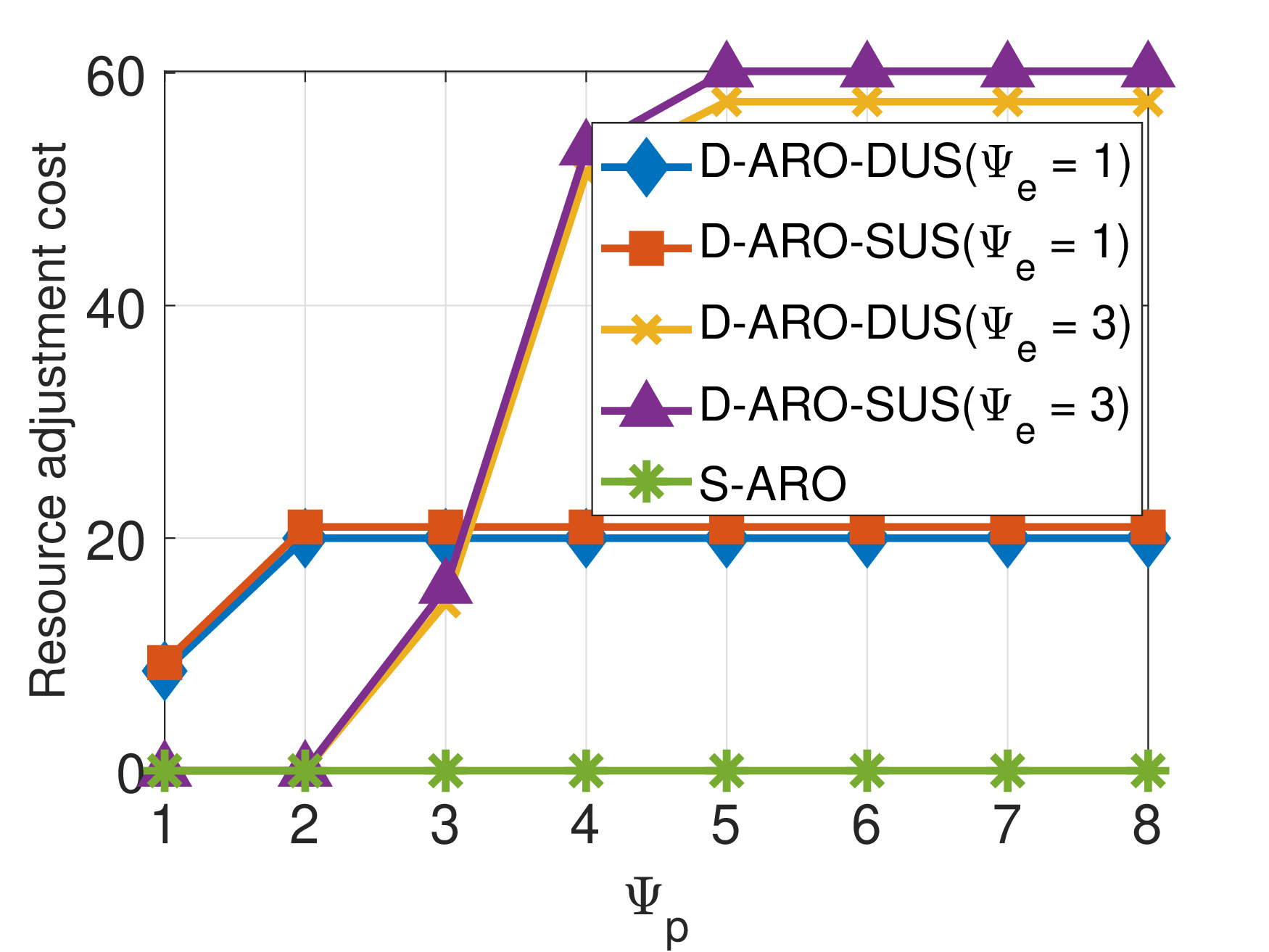}
	     \label{fig: varying_P_model_adjust}
	} \hspace*{-2.1em} 
         \subfigure[Resource adjustment: varying $e$]{
	    \includegraphics[width=0.246\textwidth,height=0.10\textheight]{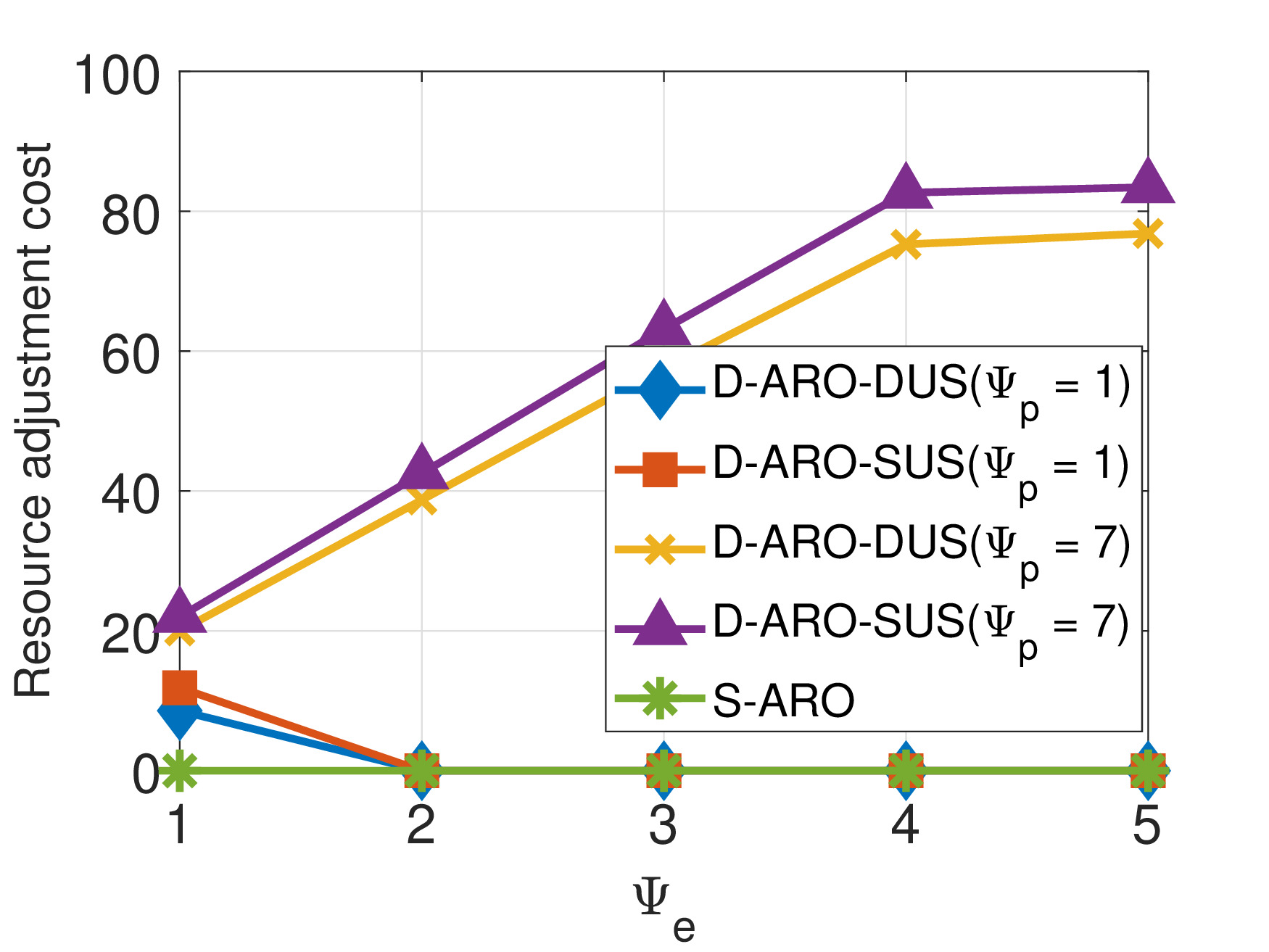}
	     \label{fig: varying_e_adjust}
	} 
	\caption{The impacts of cost parameter}
		\vspace{-0.2cm}
\end{figure} 

Figs. \ref{fig: varying_P_model_adjust} - \ref{fig: varying_e_adjust} illustrate that the \textit{S-ARO} model incurs no cost for resource adjustment, as the SP is required to reserve all necessary resources prior to the operational stage. Notably, even when the on-spot price $e$ is similar to the reserved price $p$, the SP continues to favor making decisions dynamically. However, when the on-spot price ($e$) significantly exceeds the reserved price ($p$), this incentive for adjusting resources may diminish. Consequently, the cost associated with resource adjustment in such scenarios begins to align more closely with that of the static model, as depicted in Fig. \ref{fig: varying_e_adjust}.

\noindent
\textit{\textbf{2) Impacts of uncertain parameters:}} 
$\alpha$ represents the maximum forecast error, defined as $\alpha = \max \frac{|\bar{\lambda}_i^{t} - \lambda_i^{t}|}{\lambda_i^{t}}$. 
Figs. \ref{fig: varying_alpha_cost} - \ref{fig: varying_alpha_payment} 
explore the potential for cost reduction through the dynamic uncertainty set. We maintain the settings of our proposed model and manually adjust the demand deviation for the \textit{D-ARO-SUS} and \textit{S-ARO} models using $\alpha$, aiming to examine the performance disparity between \textit{D-ARO-DUS} and its benchmarks. By altering $\alpha$ across a spectrum of $[5\%, 35\%]$, we can observe that our proposed model achieves superior performance compared to the benchmarks when $\alpha$ is set at 10\%, and it performs similarly with $\alpha$ at 5\%, as highlighted in the red circle. This enhanced performance is attributed to the model's integration of spatial-temporal dynamics in the uncertainty set, effectively lessening the conservativeness of the solution. A higher $\alpha$ indicates more significant demand variability. In contrast, the \textit{S-ARO} model accounts for the worst-case scenario (such as peak hours) from the outset of the planning period. This leads to a more conservative strategy to reliably manage peak-hour demands, even though the actual demand might not escalate to that extreme scenario.

\vspace{-0.2cm}
\begin{figure}[h!]
	\subfigure[Total cost: varying $\alpha$]{
    \includegraphics[width=0.246\textwidth,height=0.10\textheight]{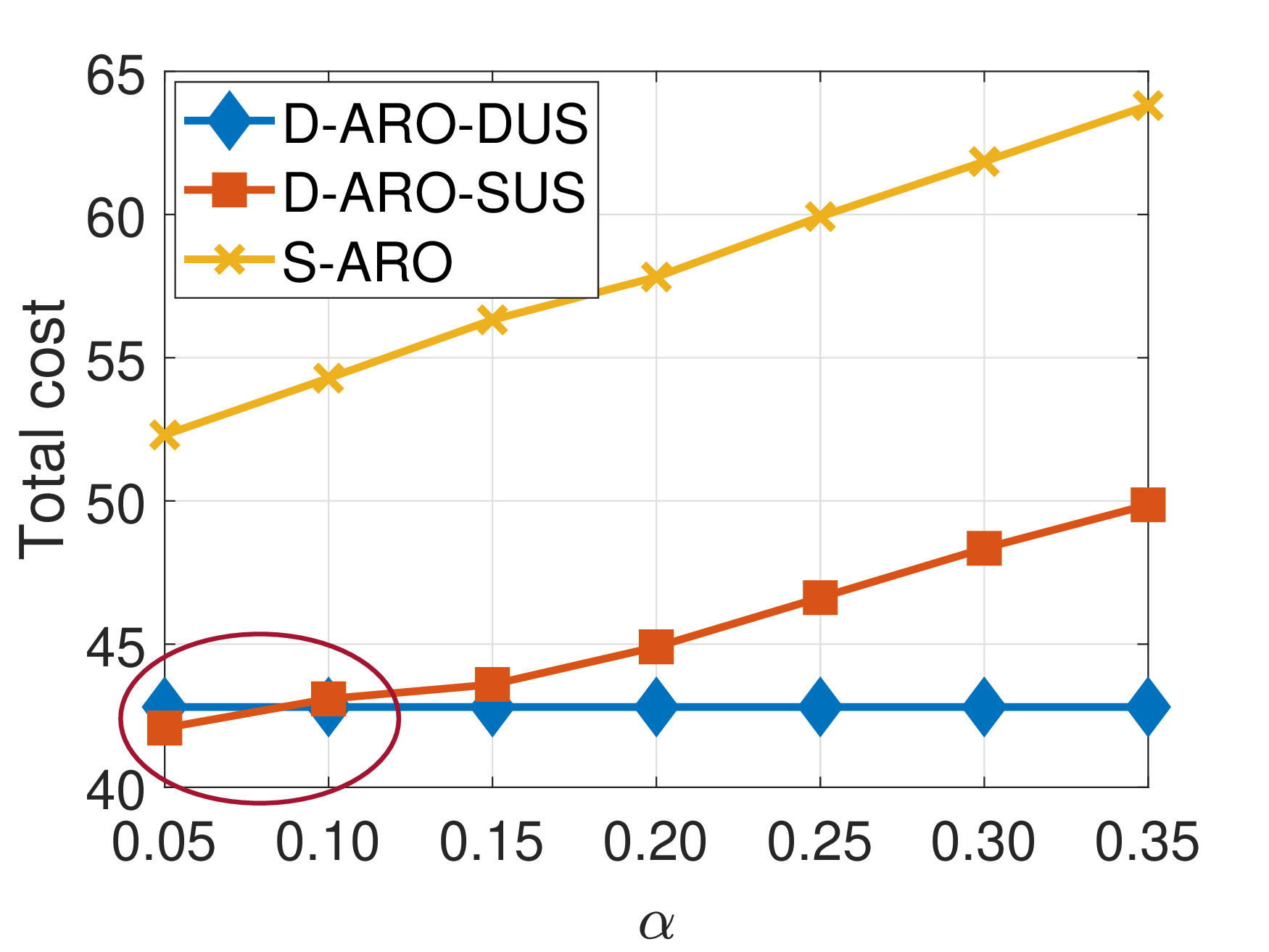}
	    \label{fig: varying_alpha_cost}
	}   \hspace*{-2.1em} 
		 \subfigure[Payment: varying $\alpha$]{
	    \includegraphics[width=0.246\textwidth,height=0.10\textheight]{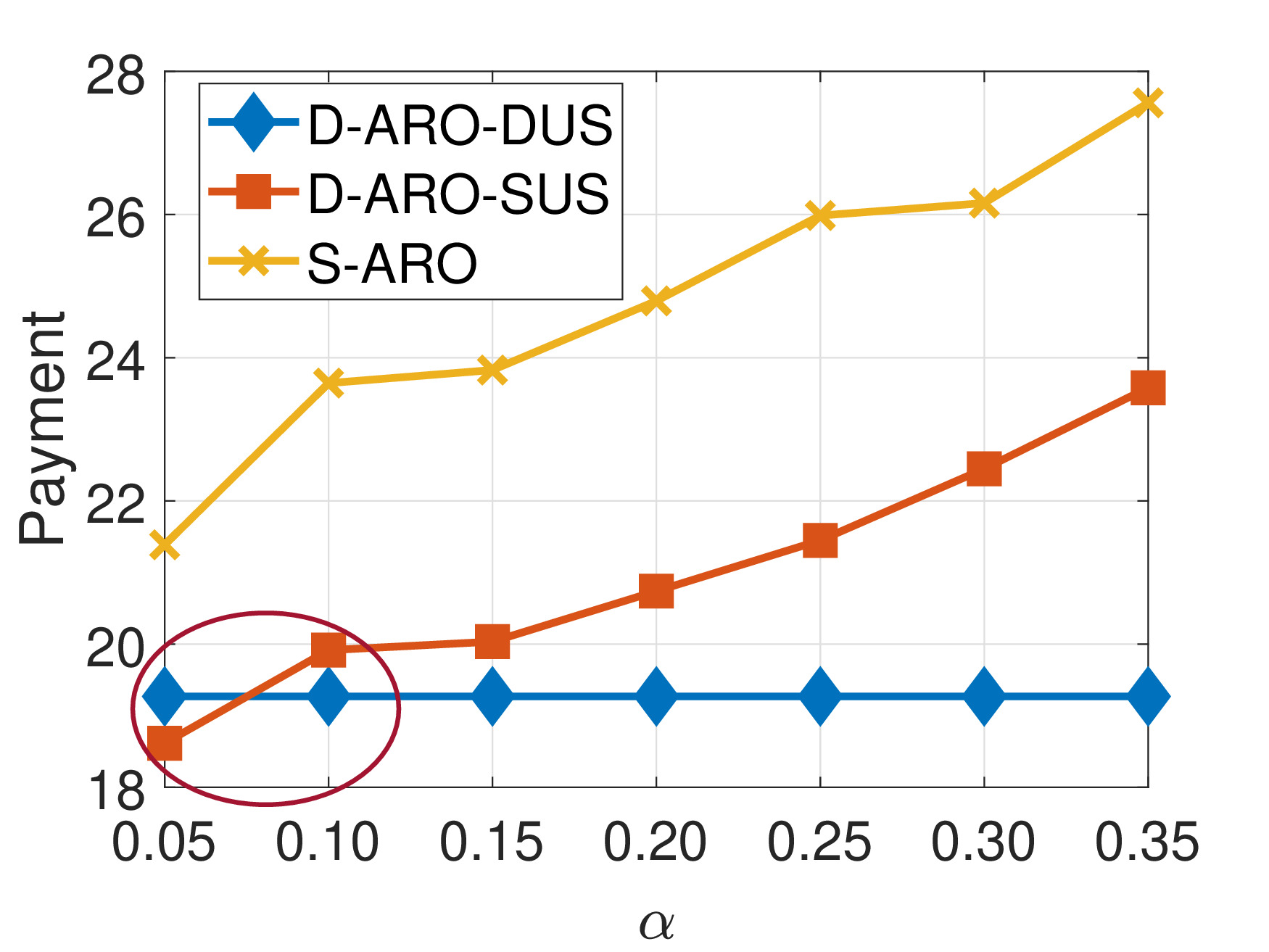}
	     \label{fig: varying_alpha_payment}
	}  \vspace{-0.2cm}
	\caption{The impacts of uncertain parameters}
		\vspace{-0.3cm}
\end{figure}

\vspace{-0.5cm}
\section{Conclusions}
\label{summary}

This paper introduced a new two-stage multi-period dynamic edge service placement framework designed to optimize edge service provisioning and allocation decisions.  
The framework 
operates in two stages: the first stage emphasizes resource reservation, while the second stage addresses dynamic service placement and workload allocation as uncertainties are unveiled. 
The proposed model integrates spatial-temporal correlation into a multi-period dynamic uncertainty set,  resulting in less conservative solutions. 
 Moreover, the model permits dynamic resource adjustments and service placement during real-time operation, enhancing adaptability to fluctuating demands. To address the proposed 
robust problem with integer recourse decisions, we developed a novel iterative \textit{ROD} algorithm that decomposes the problem into inner and outer-level components. 
Numerical analyses underscore the advantages of incorporating spatial-temporal correlations of uncertain demand, and the benefits of dynamic service placement and resource adjustments.
\vspace{-0.3cm}

\bibliographystyle{IEEEtran}

\begin{thebibliography}{10}
\providecommand{\url}[1]{#1}
\csname url@samestyle\endcsname
\providecommand{\newblock}{\relax}
\providecommand{\bibinfo}[2]{#2}
\providecommand{\BIBentrySTDinterwordspacing}{\spaceskip=0pt\relax}
\providecommand{\BIBentryALTinterwordstretchfactor}{4}
\providecommand{\BIBentryALTinterwordspacing}{\spaceskip=\fontdimen2\font plus
\BIBentryALTinterwordstretchfactor\fontdimen3\font minus \fontdimen4\font\relax}
\providecommand{\BIBforeignlanguage}[2]{{%
\expandafter\ifx\csname l@#1\endcsname\relax
\typeout{** WARNING: IEEEtran.bst: No hyphenation pattern has been}%
\typeout{** loaded for the language `#1'. Using the pattern for}%
\typeout{** the default language instead.}%
\else
\language=\csname l@#1\endcsname
\fi
#2}}
\providecommand{\BIBdecl}{\relax}
\BIBdecl

\bibitem{wshi16}
W.~Shi, J.~Cao, Q.~Zhang, Y.~Li, and L.~Xu, ``Edge computing: Vision and challenges,'' \emph{IEEE Internet Things J.}, vol.~3, no.~5, pp. 637--646, 2016.

\bibitem{RO}
A.~Ben-Tal, L.~El~Ghaoui, and A.~Nemirovski, \emph{Robust optimization}.\hskip 1em plus 0.5em minus 0.4em\relax Princeton university press, 2009.

\bibitem{Duong_iot}
D.~T. Nguyen, H.~T. Nguyen, N.~Trieu, and V.~K. Bhargava, ``Two-stage robust edge service placement and sizing under demand uncertainty,'' \emph{IEEE Internet Things J.}, vol.~9, no.~2, pp. 1560--1574, 2022.

\bibitem{resilientEC}
J.~Cheng, D.~T. Nguyen, and V.~K. Bhargava, ``Resilient edge service placement under demand and node failure uncertainties,'' \emph{IEEE Trans. Netw. Serv. Manag.}, pp. 1--1, 2023.

\bibitem{yu2021short}
X.~Yu, L.~Sun, Y.~Yan, and G.~Liu, ``A short-term traffic flow prediction method based on spatial--temporal correlation using edge computing,'' \emph{Computers \& Electrical Engineering}, vol.~93, p. 107219, 2021.

\bibitem{zhang2021cloudlstm}
C.~Zhang, M.~Fiore, I.~Murray, and P.~Patras, ``Cloudlstm: A recurrent neural model for spatiotemporal point-cloud stream forecasting,'' in \emph{Proceedings of the AAAI Conference on Artificial Intelligence}, vol.~35, no.~12, 2021, pp. 10\,851--10\,858.

\bibitem{CCGARO}
B.~Zeng and L.~Zhao, ``Solving two-stage robust optimization problems using a column-and-constraint generation method,'' \emph{Operations Research L.}, vol.~41, no.~5, pp. 457--461, 2013.

\bibitem{static16}
M.~Mechtri, C.~Ghribi, and D.~Zeghlache, ``A scalable algorithm for the placement of service function chains,'' \emph{IEEE Trans. Netw. Serv. Manag.}, vol.~13, no.~3, pp. 533--546, 2016.

\bibitem{Yang2018}
B.~Yang, W.~K. Chai, Z.~Xu, K.~V. Katsaros, and G.~Pavlou, ``Cost-efficient nfv-enabled mobile edge-cloud for low latency mobile applications,'' \emph{IEEE Trans. Netw. Serv. Manag.}, vol.~15, no.~1, pp. 475--488, 2018.

\bibitem{Lcai20}
Y.~Zhang, X.~Lan, J.~Ren, and L.~Cai, ``Efficient computing resource sharing for mobile edge-cloud computing networks,'' \emph{IEEE/ACM Trans. Netw.}, vol.~28, no.~3, pp. 1227--1240, 2020.

\bibitem{Niyato20}
M.~Karimzadeh-Farshbafan, V.~Shah-Mansouri, and D.~Niyato, ``A dynamic reliability-aware service placement for network function virtualization (nfv),'' \emph{IEEE J. Sel. Areas Commun.}, vol.~38, no.~2, pp. 318--333, 2020.

\bibitem{zhang13}
Q.~Zhang, Q.~Zhu, M.~F. Zhani, R.~Boutaba, and J.~L. Hellerstein, ``Dynamic service placement in geographically distributed clouds,'' \emph{IEEE J. Sel. Areas Commun.}, vol.~31, no.~12, pp. 762--772, 2013.

\bibitem{He19}
S.~Pasteris, S.~Wang, M.~Herbster, and T.~He, ``Service placement with provable guarantees in heterogeneous edge computing systems,'' in \emph{Proc. IEEE INFOCOM}, 2019, pp. 514--522.

\bibitem{Ouyang18}
T.~Ouyang, Z.~Zhou, and X.~Chen, ``Follow me at the edge: Mobility-aware dynamic service placement for mobile edge computing,'' \emph{IEEE J. Sel. Areas. Commun.}, vol.~36, no.~10, pp. 2333--2345, 2018.

\bibitem{Cheng_bandit}
J.~Cheng, D.~T.~A. Nguyen, L.~Wang, D.~T. Nguyen, and V.~K. Bhargava, ``A bandit approach to online pricing for heterogeneous edge resource allocation,'' \emph{Proc. IEEE NetSoft}, 2023.

\bibitem{jia18}
Y.~Jia, C.~Wu, Z.~Li, F.~Le, and A.~Liu, ``Online scaling of nfv service chains across geo-distributed datacenters,'' \emph{IEEE/ACM Trans. Netw.}, vol.~26, no.~2, pp. 699--710, 2018.

\bibitem{niyato12}
S.~Chaisiri, B.-S. Lee, and D.~Niyato, ``Optimization of resource provisioning cost in cloud computing,'' \emph{IEEE Trans. Serv. Comput}, vol.~5, no.~2, pp. 164--177, 2011.

\bibitem{Far21}
S.~Mireslami, L.~Rakai, M.~Wang, and B.~H. Far, ``Dynamic cloud resource allocation considering demand uncertainty,'' \emph{IEEE Trans. Cloud Comput.}, vol.~9, no.~3, pp. 981--994, 2019.

\bibitem{Grosu20}
H.~Badri, T.~Bahreini, D.~Grosu, and K.~Yang, ``Energy-aware application placement in mobile edge computing: A stochastic optimization approach,'' \emph{IEEE Trans. Parallel Distrib. Syst.}, vol.~31, no.~4, pp. 909--922, 2019.

\bibitem{jli21}
J.~Li, W.~Liang, and Y.~Ma, ``Robust service provisioning with service function chain requirements in mobile edge computing,'' \emph{IEEE Trans. Netw. Serv. Manag.}, vol.~18, no.~2, pp. 2138--2153, 2021.

\bibitem{nguyen20}
M.~Nguyen, M.~Dolati, and M.~Ghaderi, ``Deadline-aware sfc orchestration under demand uncertainty,'' \emph{IEEE Trans. Netw. Serv. Manag.}, vol.~17, no.~4, pp. 2275--2290, 2020.

\bibitem{yche21}
Y.~Chen, B.~Ai, Y.~Niu, H.~Zhang, and Z.~Han, ``Energy-constrained computation offloading in space-air-ground integrated networks using distributionally robust optimization,'' \emph{IEEE Trans. Veh. Technol}, vol.~70, no.~11, pp. 12\,113--12\,125, 2021.

\bibitem{fwei23}
F.~Wei, S.~Qin, G.~Feng, Y.~Sun, J.~Wang, and Y.-C. Liang, ``Hybrid model-data driven network slice reconfiguration by exploiting prediction interval and robust optimization,'' \emph{IEEE Trans. Netw. Serv. Manag.}, vol.~19, no.~2, pp. 1426--1441, 2022.

\bibitem{cheng_ddu}
J.~Cheng, D.~T.~A. Nguyen, N.~Trieu, and D.~T. Nguyen, ``Delay-aware robust edge network hardening under decision-dependent uncertainty,'' \emph{arXiv preprint arXiv:2407.06142}, 2024.

\bibitem{McCormick76}
G.~P. McCormick, ``Computability of global solutions to factorable nonconvex programs: Part i—convex underestimating problems,'' \emph{Math. Program.}, vol.~10, no.~1, pp. 147--175, 1976.

\bibitem{FairRO_EC}
D.~T.~A. Nguyen, J.~Cheng, N.~Trieu, and D.~T. Nguyen, ``A fairness-aware attacker-defender model for optimal edge network operation and protection,'' \emph{IEEE Networking Letters}, vol.~5, no.~2, pp. 120--124, 2023.

\bibitem{data1}
Y.~Li, A.~Zhou, X.~Ma, and S.~Wang, ``Profit-aware edge server placement,'' \emph{IEEE Internet Things J.}, vol.~9, no.~1, pp. 55--67, 2022.

\bibitem{data2}
S.~Wang, Y.~Guo, N.~Zhang, P.~Yang, A.~Zhou, and X.~Shen, ``Delay-aware microservice coordination in mobile edge computing: A reinforcement learning approach,'' \emph{IEEE Trans. Mob. Comput.}, vol.~20, no.~3, pp. 939--951, 2021.

\bibitem{data3}
Y.~Guo, S.~Wang, A.~Zhou, J.~Xu, J.~Yuan, and C.-H. Hsu, ``User allocation-aware edge cloud placement in mobile edge computing,'' \emph{Software: Practice and Experience}, vol.~50, no.~5, pp. 489--502, 2020.

\bibitem{Time_series}
G.~C. Reinsel, \emph{Elements of multivariate time series analysis}.\hskip 1em plus 0.5em minus 0.4em\relax Springer Science \& Business Media, 1997.

\end{thebibliography}

\newpage
\section{Appendix}
\setcounter{page}{1}



\subsection{Static ARO (S-ARO)}
\label{S_ARO_section}
We consider static adjustable robust optimization (S-ARO) over a time horizon, which is the proposed model from previous work in \cite{Duong_iot}. In this model, If SPs want to place the service onto some ENs that do not have the service installed at the beginning of the scheduling horizon, they need to decide to download the requested applications from either ENs or remote cloud with installed applications. SPs also purchase the corresponding computing resource based on the forecast traffic demand and capacity on each EN. Given the service placement and resource procurement decision, SPs cannot change any service placement decisions during the scheduling horizon. Once the service is placed at the beginning of the scheduling horizon, SP does not need to pay the fees resulting from service placement for the rest of time period until they delete this application for further operations. The problem formulation of benchmark (\textit{S-ARO}) in $\mathcal{P}_2$ is shown as follows:
\begin{subequations}
\label{S-ARO}
\begin{align}
& \underset{z,s}{\text{min}} ~ \delta \bigg\{ \sum_{j,t} p_{j}^{t} s_{j}^{t} + p_{0}^{t} s_{0}^{t} \bigg\} + \sum_{j,t} \varsigma_{j}^{t}z_{j}^{t} \nonumber \\
& + \underset{\lambda \in \mathcal{D}_{1}}{\text{max}} \underset{x,y}{\text{min}} \sum_{i,t} c_{i,0} x_{i,0}^{t} + \sum_{i,j} c_{i,j} x_{i,j}^{t}  \nonumber\\
& \text{s.t.} ~~ x_{i,0}^{t} + \sum_i x_{i,j}^{t} \geq \lambda_{i}^{t}, \: \forall i,t \\
& w \sum_{i} x_{i,j}^{t} \leq s_j^{t} z_j^{t}, ~ \forall j,t \\
& w \sum_{i} x_{i,0}^{t} \leq  s_0^{t}, ~ \forall t \\
& z_{j}^{t} = z_{j}^{t - 1}; ~ s_j^{t} = s_j^{t-1}, \forall j,t,~ \forall j,t; ~~ s_0^{t} = s_0^{t-1}, \forall t\\
& x_{i,j}^{t} = x_{i,j}^{t - 1}, \forall i,j,t;  x_{i,0}^{t} =  x_{i,0}^{t - 1}, \forall i,t\\
& z_{j}^{t} \in \{0,1\}, \forall j,t; ~ x_{i,0}^{t} \geq 0, \forall i,t ;~ x_{i,j}^{t}\geq 0,  \forall i,j,t\\
& \lambda_{i}^{t} \in \mathcal{D}_{1}(\lambda_{i}^{t}), \: \forall i,t,
\end{align}
\end{subequations}
where $\varsigma_{j}^{t} = f_j^t(1 - z_j^{t-1})z_j^{t-1} + f_{j,t}^{\sf s}$. 

\end{document}